\documentclass{amsart}[12pt]
\usepackage{graphicx,url,amsthm,xspace,enumitem,subfigure,amssymb}
\usepackage[dvipsnames]{xcolor}

\usepackage{amscd, amsmath}
\usepackage{multirow, overpic}
\usepackage{tikz}
\usetikzlibrary{decorations.pathreplacing}
\usepackage{mathabx}

\usepackage{xy}
\xyoption{all}

\usepackage{pinlabel}
\usepackage[linktocpage]{hyperref}

\usepackage[boxruled]{algorithm2e}

\hypersetup{
  colorlinks,
  citecolor=Purple,
  linkcolor=Black,
  urlcolor=Green}
  
\usepackage[margin=3cm, marginpar=2.5cm]{geometry}

\newcommand{\C}{\mathbb{C}}

\newcommand{\R}{\mathbb{R}}
\newcommand{\M}{\mathcal{M}}
\newcommand{\J}{\mathcal{J}}
\newcommand{\D}{\mathbb{D}}
\newcommand{\F}{\mathbb{F}}
\newcommand{\Q}{\mathbb{Q}}
\newcommand{\A}{\mathcal{A}}
\newcommand{\Z}{\mathbb{Z}}

\newcommand{\e}{\varepsilon}

\theoremstyle{plain}
\newtheorem{theorem}{Theorem}[section]

\newtheorem{lemma}[theorem]{Lemma}
\newtheorem{proposition}[theorem]{Proposition}
\newtheorem{corollary}[theorem]{Corollary}

\theoremstyle{definition}

\newtheorem{definition}[theorem]{Definition}

\theoremstyle{remark}
\newtheorem{remark}[theorem]{Remark}

\numberwithin{equation}{section}

\theoremstyle{plain}

\title{Exact Lagrangian tori in symplectic Milnor fibers constructed with fillings}
\author[Capovilla-Searle]{Orsola Capovilla-Searle}
\address[O.\ Capovilla-Searle]{Department of Mathematics \\ University of California Davis \\ Davis\\ CA \\ U.S.A.}
\email{ocapovillasearle@ucdavis.edu}

\begin{document}
\maketitle

\begin{abstract}
 We use exact Lagrangian fillings and Weinstein handlebody diagrams to construct infinitely many distinct exact Lagrangian tori in $4$-dimensional Milnor fibers of isolated hypersurface singularities with positive modality. We also provide a generalization of a criterion for when the symplectic homology of a Weinstein $4$-manifold is non-vanishing given an explicit Weinstein handlebody diagrams.
 \end{abstract}

\section{Introduction}

Given a Weinstein handlebody diagram of a Weinstein domain $(X,\lambda, \phi)$, where $\Lambda$ is a link of Legendrian attaching spheres, one can construct a closed exact Lagrangian submanifold $\overline{\Sigma}$ of $(X,\lambda, \phi)$ by taking the union of an exact Lagrangian filling $\Sigma$ of $\Lambda$ and the core Lagrangian disks of the handles attached along $\Lambda$. For any such Lagrangian, $[\overline{\Sigma}]\in H_n(X)$ is a primitive class. It is not known whether for every closed exact Lagrangian submanifold $\Sigma$ of a Weinstein domain $(X,\lambda, \phi)$ that is primitive in homology, if there exists a Weinstein handlebody diagram of $(X,\lambda, \phi)$ so that $\Sigma$ can be built from an exact Lagrangian filling~\cite[Remark $2.22$]{Casals_Murphy}. Vianna~\cite[Section 6.2]{vianna_2017} constructed an infinite family of exact Lagrangian tori in the 4-dimensional Milnor fibers of parabolic singularities. We show that there are infinite families of exact Lagrangian tori in a larger set of Milnor fibers of isolated hypersurface singularities using fillings:

\begin{theorem}~\label{prop:sublink}
For any $p,r\geq1$, and $q\geq 3$, there exists infinitely Hamiltonian non-isotopic exact Maslov-$0$ Lagrangian tori in Milnor fibers of $T_{p,q,r}$ singularities such that the Lagrangian tori are all smoothly isotopic and primitive in homology.
\end{theorem}

Applying properties of Milnor fibers of adjacent singularities to Theorem~\ref{prop:sublink} or Vianna's result~\cite[Section 6.2]{vianna_2017} we obtain the following corollary:

\begin{corollary}~\label{cor:inftori}
Suppose that $M_f$ is the Milnor fiber of a positive modality isolated hypersurface singularity $f$, then $M_f$ contains infinitely many Hamiltonian non-isotopic exact Maslov-$0$ Lagrangian tori that are all smoothly isotopic.
\end{corollary}

Milnor fibers of isolated hypersurface singularities have been extensively studied in homological mirror symmetry, singularity theory and low-dimensional topology. An isolated hypersurface singularity is the equivalence class of the germ of a holomorphic function, $f:\C^{n+1}\rightarrow \C$, whose differential has an isolated zero at the origin. The Milnor fiber of such a singularity is the smooth manifold $M_f=f^{-1}(\e_{\delta})\cap B_{\delta}(0)$ for suitably small $\delta>0$ and $\epsilon_\delta>0$~\cite{Milnor}. The modality of a singularity $f$ can be thought of as the dimension of a parameter space covering a neighborhood of $f$ in the space of singularities after an appropriate holomorphic reparametrization. Isolated hypersurface singularities of modality $0$ and $1$ have been classified~\cite{Arnold}. Unimodular isolated singularities are either parabolic, $T_{p,q,r}$ singularities, or one of Arnold's $14$ exceptional singularities. We focus on the Milnor fibers of $T_{p,q,r}$ singularities which were shown in~\cite[Section 2]{Keating} to be exact symplectomorphic to the affine variety
 $$T_{p,q,r}=\{(x,y,z)\in \C^3~|~x^p+y^q+z^r+xyz=1\},$$ 
for $p,q,r\in \Z_{\geq1}$ and $\frac{1}{p}+\frac{1}{q}+\frac{1}{r}\leq 1$.

The only possible exact Lagrangian submanifolds of $4$-dimensional Milnor fibers of modality $0$ singularities are spheres~\cite{IUU,Abouzaid_Smith}. This is not the case for $4$-dimensional Milnor fibers of positive modality singularities. See~\cite{keating_2021} for examples of monotone Lagrangian surfaces in Milnor fibers of positive modality singularities, and more generally in Brieskorn-Pham hypersurfaces. Keating in~\cite[Theorem 5.7]{Keating} constructed an exact Lagrangian torus in the $4$-dimensional Milnor fibers of unimodular singularities using explicit Lefschetz fibrations of the $4$-dimensional affine varieties $T_{p,q,r}$. Shende, Treumann and Williams in~\cite[Example $6.3$]{Shende_treumann_williams}, outline an argument that there are in fact infinitely many distinct Lagrangian tori that are not Hamiltonian isotopic in $T_{p,q,r}$ for $p,q,r\geq 1$. These potentially distinct tori are constructed from fillings of Legendrians in distinct Weinstein handlebodies of $T_{p,q,r}$ that are symplectomorphic but not yet shown to be Weinstein homotopic. In contrast, the tori we construct in Theorem~\ref{prop:sublink} are obtained as fillings from a single Weinstein handlebody decomposition. If $(p,q,r)\in\{(3,3,3),~(2,4,4), (2,3,6)\}$, then $T_{p,q,r}$ compactifies to a del Pezzo surface (see~\cite[Proposition $5.19$]{Keating}). Vianna in~\cite[Section 6.2]{vianna_2017} used this fact to construct infinitely many distinct exact Lagrangian tori in $T_{3,3,3}, T_{2,4,4},$ and $T_{2,3,6}$ using different almost toric fibrations of the corresponding del Pezzo surfaces.  %See~\cite[Section 7.4]{keating3} and~\cite[Remark 3.17]{Hacking_Keating} for a description of $T_{p,q,r}$ as $D^*T^2$ with $(p+q+r)$ $2$-handles attached along Legendrians. 

The Lagrangian torus that Keating constructs in $T_{p,q,r}$~\cite{Keating} implies that the Fukaya category of $T_{p,q,r}$ is not split generated by any collection of vanishing cycles as it has one generator for which this is true. Analogously, one corollary of Theorem~\ref{prop:sublink} is that the wrapped Fukaya category $T_{p,q,r}$ has an infinite number of generators which are not split generated by the vanishing cycles.

We start by producing a Weinstein handlebody diagram of the $4$-dimensional Weinstein manifold $T_{p,q,r}$ using an algorithm established by Casals and Murphy~\cite{Casals_Murphy} called the affine dictionary that allows one to obtain the Weinstein handlebody diagram of a Weinstein manifold from particular Lefschetz fibrations such as the Lefschetz fibration of $T_{p,q,r}$ given by Keating. A change in convention for the ordering of vanishing cycles and the Hurwitz move between the affine dictionary and Keating's Lefschetz fibration of $T_{p,q,r}$ resulted in an incorrect handlebody diagram in Figure $22$ of~\cite{Casals_Murphy}. We provide the corrected Weinstein handlebody diagram in Figure~\ref{fig:Tpqrhandle}.

\begin{proposition}\label{prop:Tpqr} 
The Weinstein $4$-dimensional domains
$$T_{p,q,r}=\{(x,y,z)\in \C^3~|~x^p+y^q+z^r+xyz=1\}$$
have the Weinstein handlebody diagram shown in Figure~\ref{fig:Tpqrhandle} for $p,q,r\geq 0$.
\end{proposition}

We apply Legendrian Kirby calculus moves to the Weinstein handlebody diagram shown in Figure~\ref{fig:Tpqrhandle} to find another Weinstein handlebody diagram of $T_{p,q,r}$ for $p,r\geq 1$ and $q\geq 3$ that contains the Legendrian sublink $\Lambda(\beta_{22})$ given by the $(-1)$ closure of the positive braid $(\sigma_2\sigma_3\sigma_1\sigma_2)\sigma_1^2\sigma_3^2\in Br^+_4$. The $(-1)$ closure of a braid $\beta\in Br_n^+$ is the Legendrian link given by placing $\beta$ in a standard contact neighborhood of the standard max-tb Legendrian unknot in $(\R^3, \xi_{std})$. See the last diagram on the bottom right of Figure~\ref{fig:sublink2r} for a front of $\Lambda(\beta_{22})$ as a $(-1)$ closure of a positive braid. Casals and Ng~\cite[Proposition 7.4 and 7.5]{Casals_Ng} proved that $\Lambda(\beta_{22})$ has infinitely many distinct exact Lagrangian fillings of genus $1$ using Floer theoretic methods. We use their result in conjunction with one technical lemma (Lemma~\ref{lem:restricted}) to prove Theorem~\ref{prop:sublink}. 

Weinstein handlebody diagrams can also allow one to determine whether certain symplectic invariants vanish or not without computing them explicitly. Suppose that $X_{\Lambda}$ is a $4$-dimensional Weinstein manifold whose $2$-handles are attached along a Legendrian link $\Lambda$. The \textbf{symplectic homology} of a Weinstein manifold $X_{\Lambda}$ can be computed using the Legendrian contact homology differential graded algebra of $\Lambda$, denoted by $\A(\Lambda; \partial)$~\cite{BEE,Ekholm}. We extend a criterion of Leverson~\cite[Corollary 1.5]{Leverson} for the non vanishing of the symplectic homology of a Weinstein manifold given a Weinstein handlebody diagram. Symplectic homology, denoted by $S\mathbf{H}$, is a difficult invariant to compute, and this extension gives an easy way to check whether it vanishes for a slightly broader class of Weinstein manifolds.

\begin{theorem}~\label{thm:criterion}
Let $X_{\Lambda}$ be the Weinstein $4$-manifold resulting from attaching $2$-handles along a Legendrian link $\Lambda\subset (\#^m(S^1\times S^2), \xi_{std})$ with $n$ link components. If there is any sublink $\Lambda'$ with $l$ link components for $l\leq n$ such that its Legendrian differential graded algebra has a representation 
$$\rho: (\mathcal{A}(\Lambda';\Z[t_1^{\pm 1} \ldots, t_l^{\pm 1}]), \partial) \rightarrow End(V)$$
where $V$ is a vector space over $\Q$ and $\rho(t_k)=-Id$ for $k=1,\ldots, l$, then $S\mathbf{H}(X_{\Lambda})\neq0$.
\end{theorem}

An important theme in symplectic and contact topology is understanding flexibility phenomena which refers to when an object or situation is governed only by homotopy data. Flexible Weinstein manifolds abide by an $h$-principle~\cite{Cieliebak_eliashberg} and it is often difficult to determine whether a Weinstein manifold is flexible or not. Since a flexible Weinstein manifold has vanishing symplectic homology~\cite{Kai}, Theorem~\ref{thm:criterion} also gives a criterion for when a Weinstein manifold is not flexible. 

\begin{corollary}\label{cor:flexible}
Let $X_{\Lambda}$ be the Weinstein $4$-manifold resulting from attaching $2$-handles along a Legendrian link $\Lambda \subset (\#^m(S^1\times S^2), \xi_{std})$ with $n$ link components. Suppose there is a sublink $\Lambda'$ with $l$ link components for $l\leq n$, such that its Legendrian differential graded algebra has a representation 
$$\rho: (\mathcal{A}(\Lambda';\Z[t_1^{\pm 1} \ldots, t_l^{\pm 1}]), \partial) \rightarrow End(V)$$
where $V$ is a vector space over $\Q$ and such that $\rho(t_k)=-Id$ for $k=1,\ldots, l$. Then, $X_{\Lambda}$ is not a flexible Weinstein manifold. 
\end{corollary}

For an example of an application of Theorem~\ref{thm:criterion} and Corollary~\ref{cor:flexible}, see~\cite{Acu_capovilla_gadbled_marin_murphy_starkston_wu} where the authors showed that the Weinstein complement of certain divisors in toric $4$-manifolds are not flexible. 

\subsection{Outline of the article}In Section~\ref{sec:background}, we provide definitions and background material on Weinstein $4$-manifolds, exact Lagrangian fillings, and the Legendrian contact homology differential graded algebra. In Section~\ref{sec:restricted} we prove the technical lemmas that allow us to use results of~\cite{Casals_Ng} in the proof of Theorem~\ref{prop:sublink}. In Section~\ref{sec:milnor} we give the necessary background on Milnor fibers and Lefschetz fibrations and prove Theorem~\ref{prop:sublink}, Corollary~\ref{cor:inftori}, and Proposition~\ref{prop:Tpqr}. Finally in Section~\ref{sec:criterion}, we prove the new non-vanishing criterion for symplectic homology, Theorem~\ref{thm:criterion}.

\subsection{Acknowledgements}
The author thanks their adviser Lenhard Ng for many helpful conversations and suggestions. This project developed from a suggestion of Roger Casals, whom the author would also like to thank for various helpful conversations. The author also thanks James Hughes, Adam Levine, Ailsa Keating, Emmy Murphy, and Daping Weng for useful conversations. The author was first introduced to the topic of Weinstein manifolds and Lefschetz fibrations in the $2017$ Kylerec workshop. Finally, the author thanks the referee for valuable comments and suggestions. This work was supported by the NSF grants $DGE-1644868$, and $DMS-1707652$.

\section{Background}\label{sec:background}

\subsection{Weinstein $4$-manifolds}\label{sec:weinstein}

In this article we consider only Weinstein 4-manifolds which we will now briefly review. See~\cite{Cieliebak_eliashberg} for more details on Weinstein manifolds in general. Let $(X, \omega)$ be a symplectic manifold with non-empty boundary. A $1$-form $\lambda$ is a Liouville $1$-form if $d\lambda=\omega$. The vector field $Z_{\lambda}$ that is $\omega$-dual to $\lambda$ ($\iota_{Z_{\lambda}}\omega=\lambda$) is called the Liouville vector field. A domain $(X,\lambda, Z_{\lambda})$ is a Liouville domain if the Liouville vector field $Z_{\lambda}$ is complete and transverse to the boundary $\partial X$, and points outward along $\partial X$. For a Liouville domain $(X, \lambda,Z)$, there is an induced contact structure on $\partial X$ given by $ker(\lambda|_{\partial X})$. The completion of a Liouville domain $(X, \lambda, Z_{\lambda})$ is the non-compact manifold constructed from $X$ by attaching to the boundary of $X$ the symplectization of the boundary, $X\cup_{\partial} (\R \times \partial(X))$. A \textbf{Weinstein domain} $(X,\lambda, Z_{\lambda}, \phi)$ is a Liouville domain $(X, \lambda, Z_{\lambda})$ such that $Z_{\lambda}$ is gradient-like with respect to an exhausting Morse function $\phi: X\rightarrow \R$. A \textbf{Weinstein homotopy} on a Weinstein cobordism or manifold is a smooth family of Weinstein structures $(X_t, \lambda_t, Z_t, \phi_t)$, $t\in[0,1]$, where we allow birth-death type degenerations such that the associated Liouville structures $(X_t, \lambda_t, Z_t)$ form a Liouville homotopy. If two Weinstein domains $(X_0,\lambda_0,Z_0, \phi_0)$, and $(X_1,\lambda_1,Z_1, \phi_1)$ are Weinstein homotopic, then their completions are symplectomorphic.

The model Weinstein handle $\mathcal{H}$ of index $k$ in dimension $2n$ for $k\leq n$, $(\mathbb{D}^k\times \mathbb{D}^{2n-k}, \sum_{i=1}^ny_idx_i, Z_k, \phi_k)\subset (\R^{2n}, \omega_{std})$ is described in detail in~\cite{weinstein}. In particular, $\mathcal{H}$ has convex boundary $S^{k-1} \times \mathbb{D}^{2n-k}$ ($Z_k$ points outward to the handle), and concave boundary $\mathbb{D}^k \times S^{2n-k-1}$ ($Z_k$ points inward to the handle). The convex boundary has a natural contact structure induced by $Z_k$. Then, the components of the handle are: 
\begin{itemize}
\item the \textbf{core of the handle}, $\mathbb{D}^k \times \{0\}$, a Lagrangian submanifold of $\mathcal{H}$;
\item the \textbf{co-core of the handle} $\{0\} \times \mathbb{D}^{2n-k}$, a Lagrangian submanifold of $\mathcal{H}$;
\item the \textbf{attaching sphere} $S^{k-1} \times \{0\}$, an isotropic sphere in the convex boundary $(S^{k-1}\times \mathbb{D}^{2n-k})$;
\item the \textbf{belt sphere} $\{0\}\times S^{2n-k-1}$ in the concave boundary $(\mathbb{D}^{k}\times  S^{2n-k-1})$.
\end{itemize}

A Weinstein handlebody decomposition can then be given explicitly by attaching Weinstein handles each with locally defined Morse functions that one can assemble into a global Morse function. Weinstein handlebody decompositions only contain Weinstein $k$-handles for $k\leq n$~\cite{Eliashberg}. In fact, the Weinstein handlebody decomposition of any $2n$-dimensional Weinstein manifold $(X, \lambda, Z_{\lambda}, \phi)$, is encoded by the Legendrian submanifold $\Lambda \subset (\partial X_0, \lambda, Z_{\lambda}, \phi)$ corresponding to the attaching spheres of the critical $n$-handles, where $(X_0, \lambda, Z_{\lambda}, \phi)$ is the Weinstein subcritical domain given by attaching handles of index strictly less than $n$. 

Weinstein $4$-manifolds only admit handlebody decompositions with $0, 1$ and $2$-handles. Attaching $m$ Weinstein $1$-handles to a Weinstein $0$-handle $(\mathbb{D}^4, \omega_{std})$ gives the $4$-manifold whose boundary is $(\#^m(S^1\times S^2), \xi_{std})$. The attaching region of a $1$-handle is $S^0\times \mathbb{D}^3$. Diagrammatically, we consider $S^3$ as $\R^3$ with a point at infinity and draw the attaching region of the $1$-handle as a pair of $3$-balls that are identified with a reflection. A $2$-handle is then attached along a Legendrian knot $\Lambda \subset (\#^m(S^1\times S^2), \xi_{std})$. For a $4$-dimensional Weinstein $2$-handle attached along a null-homologous Legendrian link $\Lambda \subset (\#^m(S^1\times S^2), \xi_{std})$ the framing is given by $tb(\Lambda)-1$, where $tb(\Lambda)$ is the Thurston-Bennequin number of $\Lambda$. Although the Thurston-Bennequin number is not invariant under all Legendrian isotopies in $(\#^m(S^1\times S^2), \xi_{std})$, the canonical framing of the $2$-handle is invariant under contactomorphisms.

Gompf in~\cite[Theorem 2.2]{Gompf} showed that Legendrian links in $(\#^m(S^1\times S^2), \xi_{std})$ can be isotoped into a standard form where the attaching spheres of the $1$-handles are all vertically stacked and the Legendrian links can run through the $1$-handles but are otherwise contained to a rectangular area between the pairs of $3$-balls representing the $1$-handles. Within this rectangular area the contact structure, thanks to a contactomorphism, is $(\R^3, \xi_{std}=ker(dz-ydx))$ so we can use the front projection within this region to draw the fronts of the Legendrian links. The resulting diagram on the plane is a Weinstein handlebody diagram in Gompf standard form. See for example Figure~\ref{fig:handleslide}. The Thurston-Bennequin number of a Legendrian $\Lambda \subset (\#^m(S^1\times S^2), \xi_{std})$ can be computed from the front projection as in the case of a Legendrian link $\Lambda \subset (\R^3, \xi_{std})$. Two Legendrian links in $(\#^m(S^1\times S^2), \xi_{std})$ are related by Legendrian isotopies if and only if their front projections are related by a set of six moves~\cite{Gompf}, three of which are the usual Legendrian Reidemeister moves and three referred to as Gompf moves. Gompf move $4$ and $5$ corresponds to passing cusps and crossings through $1$-handles respectively. Gompf move $6$ corresponds to moving a Legendrian strand past the attaching region of a $1$-handle.

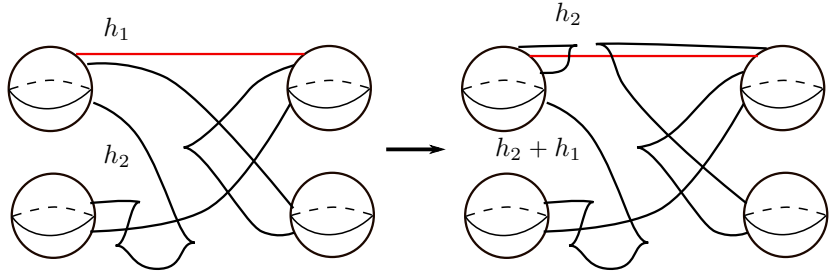
\begin{figure}\centering
\begin{tikzpicture}
\node[inner sep=0] at (0,0) {\includegraphics[width=0.7\linewidth]{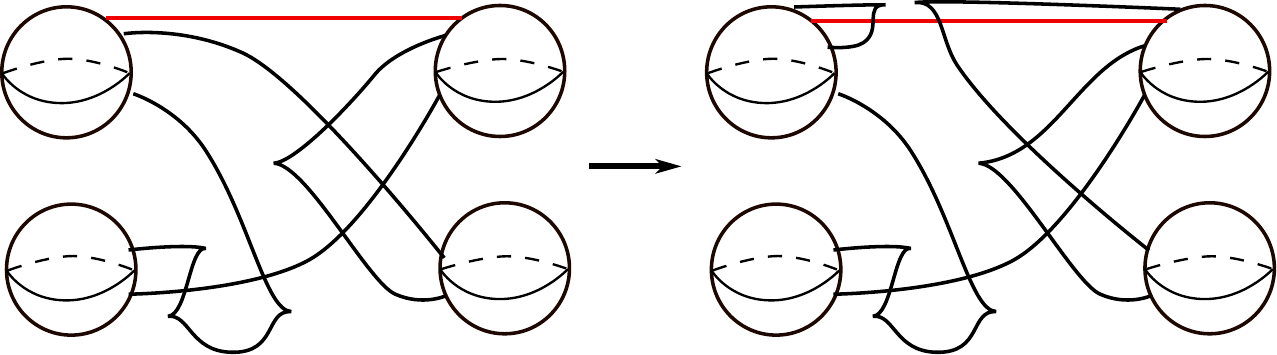}};
\node at (-4,1.7) {$h_1$};
\node at (-4,0) {$h_2$};
\node at (1.6,0.1) {$h_2+h_1$};
\node at (2,1.9) {$h_2$};
\end{tikzpicture}
\caption{A Legendrian handle slide of $h_2$ over $h_1$.} 
\label{fig:handleslide}
\end{figure}

Two Weinstein handlebody diagrams in Gompf standard form represent equivalent Weinstein domains if they are related by the following moves: isotropic or Legendrian isotopies of the attaching spheres, handle slides, handle cancellations and handle additions. If one has two Weinstein $k$-handles $h_1$ and $h_2$, a \textbf{handle slide} of $h_1$ over $h_2$ is given by isotoping the attaching sphere of $h_1$ and pushing it through the belt sphere of $h_2$, where the isotopy is an isotropic or Legendrian isotopy. In the case of a handle slide for a $4$-dimensional $1$-handle, one drags one of the $3$-balls in the attaching region of $h_1$, over $h_2$ and back following a Legendrian path. If $h_1$ and $h_2$ are $4$-dimensional Weinstein $2$-handles, with attaching spheres $\Lambda_1, \Lambda_2 \subset (\#^m(S^1\times S^2), \xi_{std})$, then the Weinstein handlebody diagram after sliding $h_1$ over $h_2$ has attaching spheres $\Lambda_1'$ and $\Lambda_2$, where $\Lambda_1'$ has a front projection given by taking a connect sum of $\Lambda_1$ with a push off of $\Lambda_2$~\cite[Proposition 1]{DingGeiges},~\cite[Proposition 2.14]{Casals_Murphy}. Figure~\ref{fig:handleslide} shows an example of a handle slide for $4$-dimensional Weinstein $2$-handles. One can add or cancel a pair of Weinstein handles $h_1$ and $h_2$ of index $k$ and $(k-1)$ respectively if the attaching sphere of $h_2$ intersects the belt sphere of $h_1$ transversely at one point~\cite[Proposition 2.17]{Casals_Murphy}. We call such a pair of handles a \textbf{canceling pair}; see Figure~\ref{fig:canceling} for an example of a canceling pair.

\begin{figure}\centering
\begin{tikzpicture}
\node[inner sep=0] at (0,0) {\includegraphics[width=0.7\linewidth]{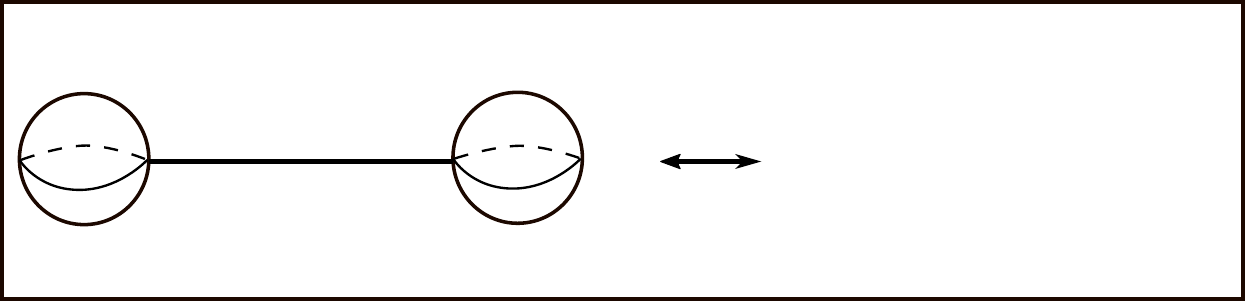}};
\end{tikzpicture}
\caption{A canceling pair of $4$-dimensional $1$ and $2$-handles.} 
\label{fig:canceling}
\end{figure}

\subsection{Exact Lagrangian surfaces}

\begin{definition}\label{defn:cobord} 
Let $(Y, ker(\alpha))$ be either $(\R^3, \xi_{std})$ or $(\#^m(S^1\times S^2), \xi_{std})$. Let $\Lambda_{+}$ and $\Lambda_{-}$ be Legendrian links in $(Y,ker(\alpha))$. An {\bf exact Lagrangian cobordism} $\Sigma$ from $\Lambda_{-}$ to $\Lambda_{+}$ is an embedded, oriented Lagrangian surface in the symplectization $(\R_t\times Y, d(e^t\alpha))$ that has cylindrical ends and is exact in the following sense: for some $N>0$, 
	\begin{enumerate}
		\item $\Sigma \cap ([-N,N]\times Y)$ is compact,
		\item $\Sigma \cap ((N,\infty)\times Y)=(N,\infty)\times \Lambda_{+}$, 
		\item $\Sigma\cap ((-\infty,-N)\times Y)=(-\infty,-N)\times \Lambda_{-}$, and 
		\item there exists a function $f: \Sigma \rightarrow \R$ and constants $\mathfrak c_\pm$ such that $e^t\alpha|_{\Sigma} = df$, where $f|_{(-\infty, -N) \times \Lambda_{-}} = \mathfrak c_{-}$, and $f|_{(N, \infty) \times \Lambda_{+}} = \mathfrak c_{+}$. We call $f$ the primitive on $\Sigma$.
	\end{enumerate}
	An {\bf exact Lagrangian filling} $\Sigma$ of a Legendrian link $\Lambda$ is an exact oriented Lagrangian cobordism from $\varnothing$ to $\Lambda$.
\end{definition}

The \textbf{Maslov number} of an exact Lagrangian surface $\Sigma$ is the greatest common divisor of the Maslov numbers of the closed loops in $\Sigma$. For a Legendrian $\Lambda$, the Maslov number of $\Lambda$ is the Maslov number of the surface $\R \times \Lambda$. We will always assume that all Legendrians and Lagrangian cobordisms are Maslov-$0$. 

\begin{remark}\label{rem:localsys}A \textbf{rank-$n$ local system} of an exact Lagrangian cobordism $\Sigma$ is a representation $\pi_1(\Sigma) \rightarrow GL(n,\F)$ where $\F$ is a field. Local systems allow one to enrich various invariants by recording homotopic data. By abuse of notation we refer to representations $\pi_1(\Sigma) \rightarrow GL(1, \Z)=\{+1, -1\}$ for orientable surfaces $\Sigma$ as (rank-$1$) local systems over $\Sigma$. Furthermore, since $GL(1,\Z)$ is abelian, it suffices to consider the maps $H_1(\Sigma;\Z)\rightarrow GL(1, \Z)$. 
\end{remark}

One can find exact Lagrangian surfaces in a Weinstein manifold using its Weinstein handlebody decomposition as follows. If $\Sigma\subset (X_0, \lambda, Z_{\lambda}, \phi) $ is an exact Maslov-$0$ Lagrangian filling of any of the critical attaching spheres $\Lambda\subset \partial X_0$, then the union of $\Sigma$ with the Lagrangian cores of the handles attached to $\Lambda$ is a closed exact Maslov-$0$ Lagrangian $\overline{\Sigma}$. See~\cite[Section 2.6]{Casals_Murphy} for some combinatorial constructions of exact Lagrangian submanifolds obtained from fillings of Legendrian attaching spheres for high dimensional Weinstein manifolds. We will use the fact that distinct exact Lagrangian fillings $\Sigma_1$ and $\Sigma_2$ can lead to distinct closed exact Lagrangian surfaces $\overline{\Sigma}_1$, and $\overline{\Sigma}_2$~\cite[Proposition 7.11]{Casals_Ng}. We first review how to distinguish fillings using Floer theory in the following subsection.

\subsection{Augmentations of the Legendrian contact differential graded algebra}\label{sec:aug}

The Legendrian contact homology differential graded algebra, $(\mathcal{A}(\Lambda;R), \partial)$, also known as the Chekanov-Eliashberg dg-algebra, is a Legendrian link invariant first defined for Legendrian links of $(\R^3, ker(dz-ydx))$ in~\cite{Chekanov} which fits into the larger SFT framework of~\cite{EGH}. For a more recent survey on this Legendrian invariant see~\cite{Etnyre_ng}. 

Let $\Lambda \subset (\#^m (S^1\times S^2), \xi_{std})$ be an oriented Legendrian link with $n$ components, $\Lambda = \Lambda^{(1)} \cup \cdots \cup \Lambda^{(n)}$ and with $k$ basepoints such that every component of $\Lambda$ has at least one basepoint decorated with a monomial $\pm t_i^{\pm}$. We will always assume that $\Lambda$ has rotation number zero. A Reeb chord $a$ of $\Lambda$ is a trajectory of the Reeb vector field which begins and ends on $\Lambda$. Let $\mathcal{R}(\Lambda)$ denote the set of Reeb chords. For a Reeb chord $a\in \mathcal{R}(\Lambda)$, let $a^+$ and $a^-$ denote the points of $\Lambda$ at the end and beginning of the Reeb chord. Then, define $r(a), c(a) \in \{1, \ldots, n\}$ as follows: $r(a)=i$ if $a^{-}\subset \Lambda^{(i)}$, and $c(a)=j$ if $a^{+}\subset \Lambda^{(j)}$. We say $a$ is a \textbf{mixed Reeb chord} if $r(a)\neq c(a)$ and a \textbf{pure Reeb chord} otherwise. We denote by $\mathcal{R}_{i, j}(\Lambda)$ the set of Reeb chords $a$ of $\Lambda$ such that $c(a)=i$ and $r(a)=j$.

\textbf{The generators of $\mathcal{A}(\Lambda;R)$:} The underlying algebra $\mathcal{A}(\Lambda;R)$ is a unital non-commutative graded algebra over the coefficient ring $R=\Z[t_1^{\pm 1}, \ldots, t_k^{\pm1}]$ that is generated by Reeb chords $a\in \mathcal{R}(\Lambda)$, and basepoints labeled by monomials $\pm t_i^{\pm}$ for $1\leq i\leq k$. There is a surjective map from $\Z[s_1^{\pm}, \ldots, s_k^{\pm}]$ to $\Z[H_1(\Lambda;\Z)]$ where each basepoint is mapped the homology class of the link component on which the basepoint lies. We assume that the basepoints $t_i^{\pm}$ commute with the Reeb chords. 

\textbf{The grading on $\A(\Lambda;R)$:} The grading is defined on the Reeb chord generators via a Conley-Zehnder index, and is defined on a word as the sum of the gradings of the letters in the word. See~\cite[Section $2.3$]{EES}, or~\cite[Section $3.1$ and Section $3.5$]{Etnyre_ng} for more details.

\textbf{The differential $\partial$:} The differential $\partial$ is defined by counting rigid $\mathcal{J}$-holomorphic disks in the symplectization of $(\#^m(S^1\times S^2), \xi_{std})$ with boundary on $\R \times \Lambda$. For Reeb chords $a, b_1 \ldots, b_m$ of $\Lambda$, denote by $\M(a;b_1, \ldots, b_m)$ the moduli space of $\J$-holomorphic disks 
$$u: (\mathbb{D}_{m+1}, \partial \mathbb{D}_{m+1}) \rightarrow (\R \times \R^3, \R\times \Lambda)$$
such that
\begin{itemize}
\item $\mathbb{D}_{m+1}$ is the disk with $m+1$ boundary points $p, q_1, \ldots, q_m$ removed and which are labeled in a counterclockwise order,
\item $u$ is asymptotic to $[0,\infty)\times a$ at $p$, and
\item $u$ is asymptotic to $(\infty, 0]\times b_i$ at $q_i$.
\end{itemize}
The quotient of $\M(a;b_1, \ldots, b_m)$ by a vertical translation of $\R$ is denoted by $\widetilde{\M}(a;b_1, \ldots, b_m)$. If $\dim(\widetilde{\M}(a;b_1, \ldots, b_m))=0$, then we say the disks $u\in \M(a;b_1, \ldots, b_m)$ are rigid. One can count such rigid $\J$-holomorphic disks over the ring $\Z[H_1(\Lambda;\Z)]$ as follows. For such a disk $u$ the boundary segment between $q_i$ and $q_{i+1}$ can be closed off on $\R \times \Lambda$ into a curve by concatenating with the capping paths from $q_i^-$ to $t_{c(q_i)}$, and the capping path from $t_{r(q_{i+1})}$ to $q_{i+1}^+$ since $c(q_i)=r(q_{i+1})$. Let $\tau_i$ denote the homology class of this curve. Then, $w(u)$ is the product of the Reeb chords and homology classes $\tau_0 \cdots \tau_m b_1\cdots b_m$. Here we are setting $b_0=a$. Moreover, if $\Lambda$ is spin then the moduli spaces of the $\J$-holomorphic curves that we are considering admit a coherent orientation~\cite{ekholm_etnyre_sullivan_Part2, Karlsson}, so one can assign a sign $sgn(u)\in \{-1,1\}$ to each rigid $\J$-holomorphic disk. For a Reeb chord $a\in \mathcal{R}(\Lambda)$, the differential is then given by
\begin{align}
\partial(a)= \sum_{\dim(\tilde{\M}(a;b_1, \ldots, b_m))=0}~\sum_{u\in \M(a;b_1,\ldots, b_m)}sgn(u)w(u).
\end{align}
Let $\partial t_i^{\pm}=0$ for any basepoint $t_i$, and extend the differential to $\A$ using the graded Leibniz rule:
$$\partial(uv)=\partial(u)v+(-1)^{|u|}u\partial(v).$$
There is a combinatorial way to compute the differential from the Lagrangian projection of $\Lambda$. See~\cite[Section 3]{Etnyre_ng} for $\Lambda\subset (\R^3, \xi_{std})$ and~\cite{Ekholm_ng} for $\Lambda\subset (\#^m (S^1\times S^2),\xi_{std})$. 

One of the main properties of the Legendrian dg-algebra is that it is functorial over exact Lagrangian cobordisms. In particular, exact Lagrangian fillings of $\Lambda$ induce augmentations $\e$ of $(\A(\Lambda; R), \partial)$~\cite{Ekholm_honda_kalman, Karlsson}.

\begin{definition} A \textbf{(graded) augmentation} of $(\A(\Lambda;R), \partial)$ to a ring of coefficients $R$ supported in grading zero is a chain map $\e: (\A(\Lambda;R), \partial)\rightarrow (R, 0)$ such that $\e(1)=1$, and for any element $a\in (\A(\Lambda; R), \partial)$ with nonzero grading $|a|\neq 0$ we have that $\e(a)=0$.
\end{definition}

\begin{theorem}[\cite{Ekholm_honda_kalman, Karlsson}]~\label{thm:induceaug}
Suppose that $\Sigma$ is a spin, oriented, embedded, Maslov-$0$ exact Lagrangian filling of the Legendrian link $\Lambda\subset (S^3, \xi_{std})$. Then $\Sigma$ induces a dg-algebra map
$$\e_{\Sigma}: (\A(\Lambda; \Z[H_1(\Sigma;\Z)]), \partial)\rightarrow (\Z[H_1(\Sigma;\Z)], 0)$$
where $\Z[H_1(\Sigma;\Z)]$ is in grading $0$. If $\Sigma$ and $\Sigma'$ are Lagrangian fillings of $\Lambda$ such that there exists a Hamiltonian isotopy from $\Sigma_1$ to $\Sigma_2$ through exact Lagrangian fillings of $\Lambda$ that fixes the boundary, then the corresponding augmentations $\e_{\Sigma}$ and $\e_{\Sigma'}$ are chain homotopic maps.
\end{theorem}

Augmented Reeb chords $\e_{\Sigma}(a)\in \Z[H_1(\Sigma;\Z)$ count the number of rigid holomorphic disks $u$ asymptotic to $a$ with boundary on $\Sigma$ and records the homology class of each disk $u$. See~\cite[Section 3.5 and Section 4]{Casals_Ng} for how to compute $\e_{\Sigma}$ explicitly for decomposable exact Lagrangian fillings. Note that not all augmentations are induced by embedded Maslov-$0$ exact Lagrangian fillings, but they are induced from immersed exact Lagrangian fillings~\cite[Theorem 1.2]{Pan_rutherford}. We only consider embedded exact Lagrangian fillings from here on out.

For an exact embedded Lagrangian filling $\Sigma$, the induced augmentation $\e_{\Sigma}$ over $\Z[H_1(\Sigma;\Z)]$ is a family of augmentations over $\Z$ where each augmentation over $\Z$ corresponds to a choice of local system of $\Sigma$:
$$\e_{\Sigma}:\A(\Lambda; \Z[H_1(\Sigma;\Z)]), \partial) \rightarrow (\Z[H_1(\Sigma;\Z)], 0)\rightarrow (\Z,0).$$

One can also enlarge the coefficient ring $\Z[H_1(\Sigma;\Z)]$ to include link automorphisms. Suppose that $\Lambda$ is a link with $n$ components and $k$ basepoints so that each link component has at least one basepoint. Let $t_i$ denote the product of basepoints on $\Lambda^{(i)}$ for $1\leq i\leq n$, so that $t_i$ is in bijection with the generator of $H_1(\Lambda^{(i)};\Z)$. Let $e_1, \ldots, e_n$ be a choice of units in $\Z$. A link automorphism of $\Lambda$ is an automorphism $\psi$ of $(\mathcal{A}(\Lambda; \Z[H_1(\Sigma;\Z)]), \partial)$ such that if $a$ is a Reeb chord of $\Lambda$ that starts at the ith link component of $\Lambda$ and ends at the jth link component, then $\psi(a)=e_{i}e^{-1}_{j}a$. Now consider $\Z[H_1(\Sigma;\Z)]\oplus \Z[t_1^{\pm}, \ldots, t_{n-1}^{\pm}]$, where $t_i= e_i/e_n$ for $1\leq i< n$. Any augmentation over $\Z[H_1(\Sigma;\Z)]$ lifts to an augmentation over $\Z[H_1(\Sigma;\Z)\oplus \Z^{n-1}]$ by setting $\tilde{\e}(a)=e_{i}e^{-1}_{j}\e(a)$. The following definition allows one to work with such augmentations.

\begin{definition}[Definition 3.9~\cite{Casals_Ng}]
A \textbf{$k$-system of augmentations} of a Legendrian link $\Lambda$ is an algebra map
$$\e: \A(\Lambda; \Z[s_1^{\pm1}, \ldots, s_k^{\pm1}]) \rightarrow \Z[s_1^{\pm1}, \ldots, s_k^{\pm1}]$$
such that $\e \circ \partial=0$, $\e(1)=1$, and $\e(a)=0$ for any $a\in \A(\Lambda;\Z[s_1^{\pm1}, \ldots, s_k^{\pm1}])$ with nonzero grading. Two $k$-systems of augmentations 
$$ \e: \A(\Lambda;\Z[s_1^{\pm1}, \ldots, s_k^{\pm1}]) \rightarrow \Z[s_1^{\pm1}, \ldots, s_k^{\pm1}]~\text{and}~ \e ': \A(\Lambda;\Z[s_1^{\pm1}, \ldots, s_k^{\pm1}]) \rightarrow \Z[{s'_1}^{\pm1}, \ldots, {s'_k}^{\pm1}]$$
are equivalent if there exists a $\Z$-algebra isomorphism $$\phi: \Z[s_1^{\pm1}, \ldots, s_k^{\pm1}] \rightarrow \Z[{s'_1}^{\pm1}, \ldots, {s'_k}^{\pm1}]$$ such that $\e '=\phi \circ \e$. The space of such isomorphisms is parametrized by $\Z_2^k \times GL_k(\Z)$.
\end{definition}

\begin{remark}\label{rem:ksystems}Let $\Sigma_1$ and $\Sigma_2$ be two spin, oriented, embedded, Maslov-$0$ exact Lagrangian fillings of an $(n+1)$ component Legendrian link $\Lambda$. By~\cite[Theorem 1.4]{CHAN}, both $\Sigma_1$ and $\Sigma_2$ realize the smooth $4$-ball genus of $\Lambda$, so $rank(H_1(\Sigma_1;\Z))=rank(H_1(\Sigma_2;\Z))=k$. Then, by Theorem~\ref{thm:induceaug} we know that they both induce $(k+n)$-systems of augmentations $\e_1: \A(\Lambda; R_1)\rightarrow R_1$, and $\e_2: \A(\Lambda; R_2)\rightarrow R_2$ where $R_i\simeq \Z[H_1(\Sigma_i;\Z)]\oplus\Z^n$. If $\Sigma_1$ and $\Sigma_2$ are Hamiltonian isotopic, then there exists a $\phi \in \Z_2^{k+n} \times GL_{k+n}(\Z)$ such that $\phi(H_1(\Sigma_1;\Z))=H_1(\Sigma_2;\Z)$. Then, $\phi\circ \e_1: \A(\Lambda; R_1)\rightarrow R_2$, and $\phi\circ\e_1$ is chain homotopic to $\e_2$ by Theorem~\ref{thm:induceaug}. See~\cite[Definition 5.5]{Capovilla_Legout_Limouzineau_Murphy_Pan_Traynor} for when two augmentations of a Legendrian link are chain homotopic when working with coefficients over $\Z$. %split dg-algebra homotopic to $\e_2$. That is, there exists an operator $K$ as in Definition~\ref{dg-algebra-homotopy} for $R=R_2$ such that for all Reeb chords $a\in \mathcal{R}^{ij}(\Lambda)$$$\alpha_i\phi \circ \e_1(a)-\alpha_j\e_2(a)=K\circ \partial(a).$$
\end{remark}

Systems of augmentations induced by fillings that are not equivalent up to linear transformations can be used to distinguish exact Lagrangian fillings~\cite{Casals_Ng}.

%\begin{remark}\label{rem:cycle1} Suppose $\e_1$ and $\e_2$ are two $k$ systems of augmentations of $\Lambda$.
%As in Remark~\ref{rem:cycle} if $a$ is a Reeb chord of $\Lambda$ such that $\partial a=0$ and there does not exist any $\phi \in \Z_2^k \times GL_k(\Z)$ such that
%$$\phi \circ \e_1(a)=\frac{\alpha_j}{\alpha_i}\e_2(a)$$
%for any units $\alpha_i, \alpha_j=\pm1$, then  $\e_1$ and $\e_2$ are not split dg-algebra homotopic. In this particular set up, the fact that $\e_1$ and $\e_2$ are not equivalent as $k$-systems of augmentations also implies that they are not split dg-algebra homotopic.
%\end{remark}

\subsection{Distinguishing exact Lagrangian surfaces in Weinstein 4-manifolds with augmentations}\label{sec:laginweinstein}

Let $X_{\Lambda}$ denote the Weinstein $4$-manifold given by attaching $2$-handles to a Legendrian link $\Lambda\subset (\#^m(S^1\times S^2), \xi_{std})$. Recall that if $\Sigma$ is an exact Maslov-$0$ Lagrangian filling of $\Lambda$, then the union of $\Sigma$ with the Lagrangian cores of the handles attached to $\Lambda$ is a closed exact Maslov-$0$ Lagrangian $\overline{\Sigma}\subset X_{\Lambda}$. Let the system of augmentations induced by $\Sigma$ be denoted by
$$\e_{\Sigma}: (\A(\Lambda; \Z[H_1(\Sigma;\Z)]), \partial)\rightarrow (\Z[H_1(\Sigma;\Z)], 0).$$
In general, composing $\e_{\Sigma}$ with a local system on $H_1(\Sigma;\Z)$ results in a family of $\Z$ valued augmentations. Such a composition is well behaved under equivalence of local systems and systems of augmentations. Since $H_1(\Sigma;\Z)$ and $H_1(\overline{\Sigma};\Z)$ are not in general isomorphic (namely, if $\Lambda$ is a link with more than one component), not all local systems on $\Sigma$ extend to local systems on $\overline{\Sigma}$. We are interested in the systems of augmentations where this extension is possible.

\begin{definition}[Definition $7.9$~\cite{Casals_Ng}]\label{defn:restricted}
Let $\Sigma$ be an exact Lagrangian filling of a Legendrian link $\Lambda \subset (S^3, \xi_{std})$ that induces a system of augmentations $\e_{\Sigma}: (\A(\Lambda; \Z[H_1(\Sigma;\Z)]), \partial)\rightarrow (\Z[H_1(\Sigma;\Z)], 0)$, where $\Lambda$ is equipped with the null-cobordant spin structure. The \textbf{restricted system of augmentations} associated to $\Sigma$ is the composition
$$\e_{\overline{\Sigma}}: (\A(\Lambda; \Z[H_1(\Sigma;\Z)]), \partial)\rightarrow (\Z[H_1(\Sigma;\Z)], 0)\rightarrow (\Z[H_1(\overline{\Sigma};\Z)], 0), $$
where the second map is induced by the quotient map $H_1(\Sigma;\Z)\rightarrow H_1(\overline{\Sigma};\Z)$
\end{definition}

Suppose one starts with the Legendrian dg-algebra of $\Lambda$ $(\A(\Lambda ; \Z[H_1(\Lambda;\Z)]), \partial)$ such that $\Lambda$ has the Lie group spin structure. To translate to the dg-algebra where $\Lambda$ has a null-cobordant spin structure we map $t_i\rightarrow -t_i$ where $t_i$ is the product of basepoints on the link component $\Lambda^{(i)}$ which represents $H_1(\Lambda^{(i)};\Z)$. Since each homology component of $\Lambda$ is null-homologous in $\overline{\Sigma}$, then any restricted augmentation defined over the null cobordant spin structure must map $t_i\rightarrow -1$. For Legendrian knots, $H_1(\Sigma;\Z)$ is isomorphic to $H_1(\overline{\Sigma};\Z)$ for null-homologous knots $\Lambda$, so all augmentations are restricted for knots.
\begin{definition}
A Legendrian link $\Lambda$ is \textbf{aug-infinite} if the collection of all $\Z$ valued augmentations $\e_{\overline{\Sigma}}(\A(\Lambda; \Z), \partial)\rightarrow (\Z, 0)$ induced by Maslov-$0$ exact Lagrangian fillings $\Sigma$ of $\Lambda$, ranging over all such fillings, is infinite. A Legendrian link $\Lambda$ is \textbf{restricted aug-infinite} if the collection of all $\Z$ valued restricted augmentations $\e_{\overline{\Sigma}}(\A(\Lambda; \Z), \partial)\rightarrow (\Z, 0)$ induced by Maslov-$0$ exact Lagrangian fillings $\Sigma$ of $\Lambda$, ranging over all such fillings, is infinite.
\end{definition}

\begin{remark}
Note that to check whether a Legendrian $\Lambda$ is restricted aug-infinite one must work with the null-cobordiant spin structure. Whereas, to check whether a Legendrian $\Lambda$ is aug-infinite one might have to also work with the Lie group spin structure. Note also that the number of basepoints on a link component does not matter, as we always consider the product of basepoints representing the homology of that link component.
\end{remark}

\begin{proposition}[Proposition $7.11$ in~\cite{Casals_Ng}]\label{prop:weinstein_different}
Let $\Lambda\subset (S^3, \xi_{std})$ be a Legendrian link and $\Sigma_1, \Sigma_2\subset (\mathbb{D}^4, \omega_{std})$ two exact Lagrangian fillings of $\Lambda$. Suppose that the two restricted systems of augmentations
$$\e_{\overline{\Sigma}_1}: \A(\Lambda)\rightarrow \Z[H_1(\overline{\Sigma}_1;\Z)], ~~\e_{\overline{\Sigma}_2}: \A(\Lambda)\rightarrow \Z[H_1(\overline{\Sigma}_2;\Z)],
$$
are not dg-algebra homotopic. Then, the exact Lagrangian surfaces $\overline{\Sigma}_1, \overline{\Sigma}_2 \subset W(\Lambda)$ are not Hamiltonian isotopic in the Weinstein $4$-manifold $X_\Lambda$.
\end{proposition}

The proof of this proposition relies on the the fact that for such a Weinstein $4$-manifold, $X_{\Lambda}$, the Wrapped Fukaya Category of $X_{\Lambda}$ is generated by the union of co-cores of the Weinstein $2$-handles of $X_{\Lambda}$~\cite{Chantraine_dimitriglou_2017}. Then, by considering the Yoneda embedding $Hom(C, -):=C\mathcal{W}(C,-)$, they show $C\mathcal{W}(C, \overline{\Sigma}_1)$ and $C\mathcal{W}(C, \overline{\Sigma}_2)$ are distinct as $C\mathcal{W}(C)$-modules and conclude that $ \overline{\Sigma}_1$ and $ \overline{\Sigma}_2$ are distinct objects in $\mathcal{W}(X_{\Lambda})$.

\section{Exact Lagrangian fillings and restricted augmentations}\label{sec:restricted}

In this section we prove that certain Legendrians are restricted aug-infinite to conclude that Weinstein manifolds whose Weinstein handlebody diagrams contain a restricted aug-infinite Legendrian sublink then have infinitely many distinct closed exact Lagrangians. We begin by upgrading the following proposition. 

\begin{proposition}[Proposition 7.5~\cite{Casals_Ng}]\label{cor:wig1}
Suppose $\Lambda_{-}$ is an aug-infinite Legendrian link in $(\R^3, \xi_{std})$. If there exists a decomposable Maslov-$0$ exact Lagrangian cobordism $\Sigma$ from $\Lambda_{-}$ to $\Lambda_{+}$, then $\Lambda_{+}$ is also aug-infinite.
\end{proposition}

Decomposable exact Lagrangian cobordisms are the concatenation of three types of elementary cobordisms~\cite{Ekholm_honda_kalman}: perturbations of Legendrian isotopy traces, minimum cobordisms (essentially the unique lagrangian disk filling of the max-tb unknot $U$~\cite{eliashberg_polterovich_1996}), and saddle cobordisms. In the following proof we use the combinatorial framework given in~\cite[Section 3.5]{Casals_Ng} to construct systems of augmentations of decomposable fillings. We will call a product of basepoints on a link component $\Lambda^{(i)}$ distinguished if they represent the generator of $H_1(\Lambda^{(i)};\Z)$.

\begin{lemma}\label{lem:restricted}
Suppose that $\Lambda_{-}, \Lambda_+ \subset (\R^3, \xi_{std})$ are Legendrians links such that there exists a decomposable exact Maslov-$0$ Lagrangian cobordism from $\Lambda_-$ to $\Lambda_+$. Then, any restricted augmentation of $\Lambda_-$ lifts to a restricted augmentation of $\Lambda_+$.
\end{lemma}

\begin{proof}

Let $\e_-$ denote a restricted augmentation induced by a filling on $\Lambda_-$ and $\e_+$ the augmentation on $\Lambda_+$ induced by concatenating with an elementary cobordism $\Sigma$ from $\Lambda_-$ to $\Lambda_+$. It suffices to check that for each link component of $\Lambda_+$ the distinguished product of basepoints is mapped to $-1$ by $\e_+$. 

If $\Sigma$ is induced by a Legendrian isotopy, then both Legendrian have the same number of link components which we can decorate with a single basepoint $t_i$ so that $\e_+(t_i)=\e_-(t_i).$ By assumption, $t_i$ is a distinguished basepoint on both Legendrians, and $\e_-$ is a restricted augmentation. So $\e_-(t_i)=-1$ and therefore $\e_+(t_i)=-1$. If $\Sigma$ is a minimum cobordism, then $\Lambda_+=U\sqcup \Lambda_-$. Decorate $\Lambda$ with $n$ basepoints so each link component is decorated with a single basepoint and $U$ is decorated with $t_0$. Then, $\e_+(t_0)=-1$, and $\e_+(t_i)=\e_-(t_i)$ for $i\neq 0$. For $i\neq 0$, $t_i$ is a distinguished basepoint on each link component of $\Lambda_-$, so $\e_+(t_i)=\e_-(t_i)=-1$. 

Suppose that $\Sigma$ is a saddle cobordism. Then, $\Lambda_+$ has one more or one fewer link component than $\Lambda_-$. Suppose it is the jth and kth link component of $\Lambda_+$ or $\Lambda_-$ which merge into one. Let $\sigma$ denote the unstable manifold on $\Sigma$ and mark with basepoints $s$, and $-s^{-1}$ the two endpoints of $\sigma$ on $\Lambda_-$. See for example~\cite[Figure 16]{Casals_ng}. Then, we have that
$$\e_+(t_j)=\e_-(t_j)s~~\text{and}~~~\e_{+}(t_k)=-\e_{-}(t_k)s^{-1},$$
where $t_j$ and $t_k$ are the basepoints on the link component that is split or merged by $\Sigma$ in $\Lambda_+$.

If $\Lambda_+$ has one fewer link component than $\Lambda_-$, then $t_jt_k$ is a distinguished basepoint on $\Lambda_+$ which represents the homology class of one of its link components. Since $\Lambda_-$ is assumed to be restricted aug-infinite, $\e_-(t_j)=\e_-(t_k)=-1$. Therefore $\e_+(t_jt_k)=-\e_-(t_j)\e_-(t_k)ss^{-1}=-1$. If $\Lambda_+$ has one more link component than $\Lambda_-$, then $\e_{-}(t_jt_k)=-1$. Choose a local system such that $\eta(s)=-\e_-(t_j)$. Then, $\eta \circ \e_+(t_j)=\e_-(t_j)\eta(s)=-1$, and $\eta \circ \e_+(t_k)=-\e_-(t_k)\eta(s^{-1})=-1$. In both cases for all other basepoints, we have $\e_+(t_i)=\e_-(t_i)=-1$. Thus, a restricted augmentation on $\Lambda_-$ lifts to a restricted augmentation on $\Lambda_+$.
\end{proof}

The following corollary is an inmediate consequence of Lemma~\ref{lem:restricted} and Proposition~\ref{cor:wig1}.

\begin{corollary}\label{cor:cobordism_restricted}
Suppose $\Lambda_{-}$ is a restricted aug-infinite Legendrian link in $(\R^3,\xi_{std})$. If there exists an exact Maslov-$0$ Lagrangian cobordism $\Sigma$ from $\Lambda_{-}$ to $\Lambda_{+}$, then $\Lambda_{+}$ is also restricted aug-infinite.
\end{corollary}

\begin{corollary}\label{cor:restrictedfamilies}
The Legendrian links $\Lambda(\beta_{ab}) \subset (\R^3,\xi_{std})$ given by the $(-1)$ closure of the positive braid $(\sigma_2\sigma_1\sigma_3\sigma_2)^4\sigma_1^a\sigma_3^b\in Br_4^+$, for $a,b\geq 1$ are restricted aug-infinite.
\end{corollary}

\begin{proof}
For $a,b>1$ there exists a Maslov-$0$ exact Lagrangian cobordism from $\Lambda(\beta_{11})$ to $\Lambda(\beta_{ab})$ constructed using saddle cobordisms. In particular, this cobordism is constructed by pinching $(a-1)$ and $(b-1)$ crossings of $\Lambda(\beta_{ab})$ to obtain $\Lambda(\beta_{11})$. Casals and Ng showed that $\Lambda(\beta_{11})$ is restricted aug-infinite in the proof of Corollary 1.7 in~\cite{Casals_Ng}. By Corollary~\ref{cor:cobordism_restricted}, we can then conclude that $\Lambda(\beta_{ab})$ are restricted aug-infinite for $a,b\geq 1$.
\end{proof}

We apply Proposition~\ref{prop:weinstein_different} and Corollary~\ref{cor:restrictedfamilies} to conclude the following corollary. The fact that the Lagrangian surfaces that we consider are smoothly isotopic follows from~\cite[Theorem 1.1 and Proposition 7.1]{Casals_Ng}.

\begin{corollary}\label{thm:weinsteinsublink}
Let $X$ be a Weinstein manifold that has a Weinstein handlebody diagram containing as a Legendrian sublink the link $\Lambda(\beta_{ab})$ for some $a,b\geq 1$. Then $X$ contains infinitely many Hamiltonian non-isotopic Lagrangian surfaces of genus $\frac{a+b-c+2}{2}$ respectively, where $2\leq c\leq 4$ is the number of link components of $\Lambda(\beta_{ab})$. These surfaces are all Maslov-$0$, smoothly isotopic and primitive in homology.
\end{corollary}

\section{Constructing infinitely many distinct Lagrangian Tori in Milnor fibers}\label{sec:milnor}

Let $f: \C^{n+1} \rightarrow \C$ be a holomorphic function such that $f(0)=0$, $df|_0=0$ and $df\neq 0$ on $B_r(0)\backslash \{0\}$ for a sufficiently small $r\in \mathbb{R}$. An \textbf{isolated hypersurface singularity} at $0$ is the equivalence class of the germ of such a holomorphic function $f$, up to bi-holomorphic changes of coordinates that fix $0$. For an isolated hypersurface singularity $f$, the \textbf{Milnor fiber} of $f$ is the smooth manifold
$$ M_f:=f^{-1}(\e_{\delta})\cap B_{\delta}(0)$$
for suitable choices of $\delta>0$ and $\e_{\delta}>0$ which depends on our choice of $\delta$. Milnor proved that $M_f$ is independent of the aforementioned choices~\cite{Milnor}. As the intersection of a hypersurface with a ball, the Milnor fiber has an exact symplectic structure inherited from $(\C^{n+1}, d(\frac{i}{4}\sum_{i=0}^n (z_id\bar{z_i} \wedge \bar{z_i}dz_i))$. Furthermore, the negative Liouville flow is given by 
the gradient flow of $h(x)$ with respect to the standard K\"{a}hler metric. Choose a cut-off function
$$h_{A}(x)=|| Ax||^2$$
for some $A\in GL_{n+1}(\C)$. Then,
$$f^{-1}(\e_{\delta}) \cap \{h_{A}(x)\leq \delta\}$$
is an exact symplectic manifold with contact type boundary since the negative gradient flow of $h$ points strictly inwards at any point of the hypersurface $|| Ax||^2= \delta$. Thus, we can describe the Milnor fiber as a Liouville domain which we can then complete by attaching cylindrical ends. Call this the completed Milnor fiber of $f$ and denote it by $M_f$. Keating~\cite[Lemmas 2.6 and 2.7]{Keating} showed that $M_f$ is independent of the choice of holomorphic representative of $f, A, \delta$, and $\e_{\delta}$ up to exact symplectomorphisms. 

Isolated hypersurface singularities are equivalent to polynomials. This follows from the fact that the $(\mu+1)$-jet of a function $f$ at an isolated critical point with Milnor number $\mu$ is sufficient. That is, for any other function $g$ with the same $(\mu+1)$-jet, there exists a biholomorphic change of coordinates between $f$ and $g$. The $k$-jet of a function is the Taylor series of the mapping truncated at degree k and deleting the constant term. The Milnor number $\mu$ of $f$ is the number of isolated non-degenerate singularities near $0$ in a morsification of $f$. 

One important property of a singularity is its modality. The group $G$ of germs of diffeomorphisms $(\C^n,0)\rightarrow (\C^n,0)$ acts on the space of function germs $f:(\C^n, 0)\rightarrow (\C, 0)$. The modality of a singularity $f$ is the least integer $m$ such that a sufficiently small neighborhood of $f$ is covered by a finite number of $m$-parameters of orbits. Moreover, this action of $G$ on the space of function germs $f:(\C^n, 0)\rightarrow (\C, 0)$ induces an action on the $k$-jet space of these function germs. In fact, the \textbf{modality} of a singularity $f$ is the modality of any of its $k$-jets for $k\geq \mu(f)+1$ where $\mu$ is the Milnor number of $f$. Isolated hypersurface singularities of modality $0$ and $1$ are classified~\cite{Arnold}. There are three families of unimodular isolated singularities: three parabolic singularities, the hyperbolic series $T_{p,q,r}$ and $14$ exceptional singularities. The hyperbolic singularities $T_{p,q,r}$ are given by 
$$x^p+y^q+z^r+axyz,$$
where $a\in \C^*$, $p,q,r\in \Z_{\geq 0}$ and $\frac{1}{p}+\frac{1}{q}+\frac{1}{r}<1$. The three parabolic singularities are the $T_{p,q,r}$ singularities such that $\frac{1}{p}+\frac{1}{q}+\frac{1}{r}=1$ (i.e. $(p,q,r)\in\{(3,3,3),(2,4,4),(2,3,6)\}$). If $f$ is a weighted homogeneous function with a single isolated singularity at $0$, then $M_f$ is exact symplectomorphic to any of the hypersurfaces $f^{-1}(\e)$ for any $\e \in \C^{*}$ equipped with the standard K\"{a}hler exact symplectic form. The $T_{p,q,r}$ singularities are not in general weighted homogeneous, but Keating~\cite[Lemma 2.16]{Keating} proved that they are independent of $a\in \C^*$, so we can set $a=1$.

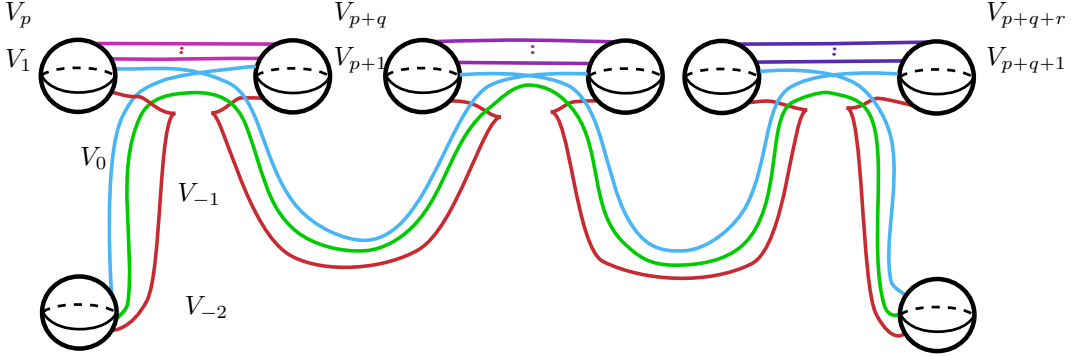
\begin{figure}[h!]
    \centering
\begin{tikzpicture}
	\node[inner sep=0] at (0,0) {\includegraphics[width=0.8\linewidth]{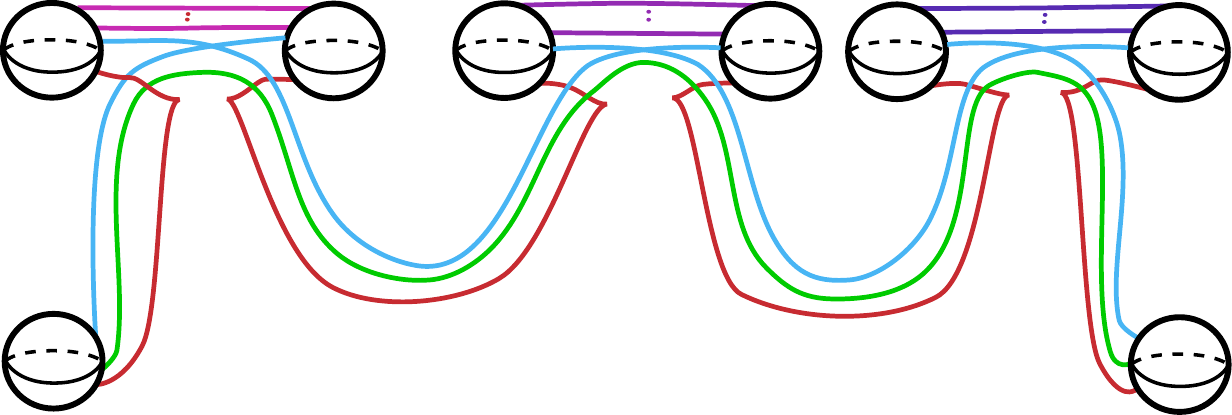}};
\node at (-4,-1.5) {$V_{-2}$};
\node at (-4.1, 0) {$V_{-1}$};
\node at (-5.5, 0.5) {$V_0$};
\node at (-6.5, 1.8){$V_1$};
\node at (-6.5, 2.4){$V_{p}$};
\node at (-1.95, 1.8){$V_{p+1}$};
\node at (-1.95, 2.4){$V_{p+q}$};
\node at (6.9, 1.8){$V_{p+q+1}$};
\node at (6.9, 2.4){$V_{p+q+r}$};
\end{tikzpicture}
\caption{A Weinstein handlebody diagram of the Milnor fiber of a $T_{p,q,r}$-singularity where $p,q,r\geq0$.}~\label{fig:Tpqrhandle}
\end{figure}

In summary, Keating~\cite{Keating} showed that the Milnor fiber of $T_{p,q,r}$ singularities are given by the affine varieties $T_{p,q,r}=\{(x,y,z)\in \C^3~|~x^p+y^q+z^r+xyz=1\}\subset \C^3$ for $p,q,r$ such that $p,q,r\in \Z_{\geq0}$ and $\frac{1}{p}+\frac{1}{q}+\frac{1}{r}\leq 1$. In~\cite{keating2} she gave a Lefschetz fibration $\Xi:T_{p,q,r} \rightarrow \C$ shown with regular fiber $\Xi^{-1}(0)$ symplectomorphic to a thrice punctured torus and $(p+q+r+3)$ distinct critical values which we will use to construct a Weinstein handlebody diagram of $T_{p,q,r}$. We will now briefly review Lefschetz fibrations for $4$-manifolds. See~\cite{Seidel} for a comprehensive treatment of Lefschetz fibrations.

\begin{definition}
Let $\D^2 \subset \C$ denote the closed unit disk, and let $X$ be a compact complex manifold with corners of real dimension $4$. A Lefschetz fibration $\pi: X\rightarrow \D^2$ is as follows:
\begin{itemize}
\item a map $\pi$ that is a proper smooth fibration except for a finite number of critical points $crit(\pi)$ in the interior of $X$;
\item near each critical point $p\in crit(\pi)$, there are local holomorphic coordinates in which $\pi$ is given by $(z_1, z_2)\rightarrow \pi(p)+z_1^2+z_2^2$. All critical values are distinct;
\item each smooth fiber of $\pi$ is a Liouville domain.
\end{itemize}
\end{definition}

Consider a Lefschetz fibration on an affine variety $\pi: X \rightarrow \C$, denote the regular fiber $\pi^{-1}(0)$ by $F_{\pi}$ and critical values by $c_1, \ldots, c_s$ in the interior of $\mathbb{D}^2$. For every critical value $c_i$, there is an associated \textbf{vanishing cycle} $V_i\subset F_{\pi}$, which is the boundary of an embedded Lagrangian disk $\Delta_i\subset X$. The image of the Lagrangian embedded disk $\pi(\Delta_i)$ is an embedded path $\gamma_i:[0,1]\rightarrow \mathbb{D}^2$ with endpoints $\gamma_i(0)=0$ and $\gamma_i(1)=c_i$. The vanishing cycle $V_i$ is then $\Delta_i\cap F_{\pi}$, a Lagrangian sphere in the regular fiber. A Lefschetz fibration $\pi:X\rightarrow \C$ can be recuperated from the data of the regular fiber $F_{\pi}$ and the ordered set of vanishing cycles, whose ordering is determined by the cyclic ordering of the vanishing paths. 

Giroux and Pardon~\cite[Theorem 1.5]{giroux_pardon_2017} proved that any $2n$-dimensional Weinstein manifold $(X, \lambda, Z, \phi)$ admits a Lefschetz fibration $\pi: (X, \lambda, Z, \phi)\rightarrow \C$. In the case that $X$ is an affine variety it is simple to construct a Lefschetz fibration up to isotopy by using a generic hyperplane section~\cite{mclean_2009}. In fact, one can obtain a Weinstein handle-decomposition of $X$ from the data of the Lefschetz fibration, that is from the regular fiber $F_{\pi}$ and the set of vanishing cycles $\{V_1, \ldots, V_s\}$. We will now describe this process for $4$-dimensional Weinstein manifolds. Start with the Weinstein domain $F_{\pi}\times \mathbb{D}^2$ with contact boundary $F_{\pi}\times S^1\subset \partial(F_{\pi}\times \mathbb{D}^2)$ and attach Weinstein $2$-handles along the Legendrian lifts of the exact Lagrangian vanishing cycles $V_{1}, \ldots, V_n$. The ordering of the vanishing cycles determines the relative Reeb height of the Legendrian lifts. Thus, the data of a Lefschetz fibration, a regular fiber $F$ and an ordered set of vanishing cycles $\{V_1, \ldots, V_s\}$, determines a Weinstein domain which we denote by $lf(F; V)$. Two Weinstein manifolds $lf(F; V)$, and $lf(F'; V')$ are Weinstein equivalent if $F'$ is Weinstein equivalent to $F$ and the vanishing cycles $V'$ are Hamiltonian isotopic to those the vanishing cycles $V$, up to a cyclic shift. Moreover, there are two additional Weinstein equivalences given by changing the Lefschetz fibration by stabilization or Hurwitz moves. Hurwitz moves are given by:
\begin{align*}
lf(F, \{V_1, \ldots, V_i, V_{i+1}, \ldots, V_s\})&=lf(F, \{V_1, \ldots, V_{i+1}, \tau_{V_{i+1}}(V_i),\ldots, V_s\}),~\text{and}\\
lf(F, \{V_1, \ldots, V_i, V_{i+1}, \ldots, V_s\})&=lf(F, \{V_1, \ldots, \tau^{-1}_{V_{i}}(V_{i+1}),V_i,\ldots, V_s\}),
\end{align*}
where $\tau$ is a Dehn twist. 

Casals and Murphy~\cite[Recipe 3.3]{Casals_Murphy} establish an algorithm they call the affine dictionary which inputs certain Lefschetz fibrations $lf(F;V)$ of complex affine varieties and outputs a Weinstein handlebody diagram of $lf(F;V)$ by providing the fronts of the Legendrian lifts of the vanishing cycles $V$. We apply the affine dictionary to Keating's Lefschetz fibration $\Xi:T_{p,q,r} \rightarrow \C$ in order to produce a Weinstein handlebody diagram of $T_{p,q,r}$. Note that we are using the conventions for Lefschetz fibrations from~\cite{Casals_Murphy} where we order vanishing cycles counterclockwise and also use the Hurwitz move convention in~\cite{Casals_Murphy} where $\{V_i ,~V_{i+1} \}\rightarrow \{\tau^{-1}_{V_i} V_{i+1} ,~V_{i}\}$.

\begin{figure}[h!]
\centering
	\begin{tikzpicture}
	\node[inner sep=0] at (0,0) {\includegraphics[width=8
	cm]{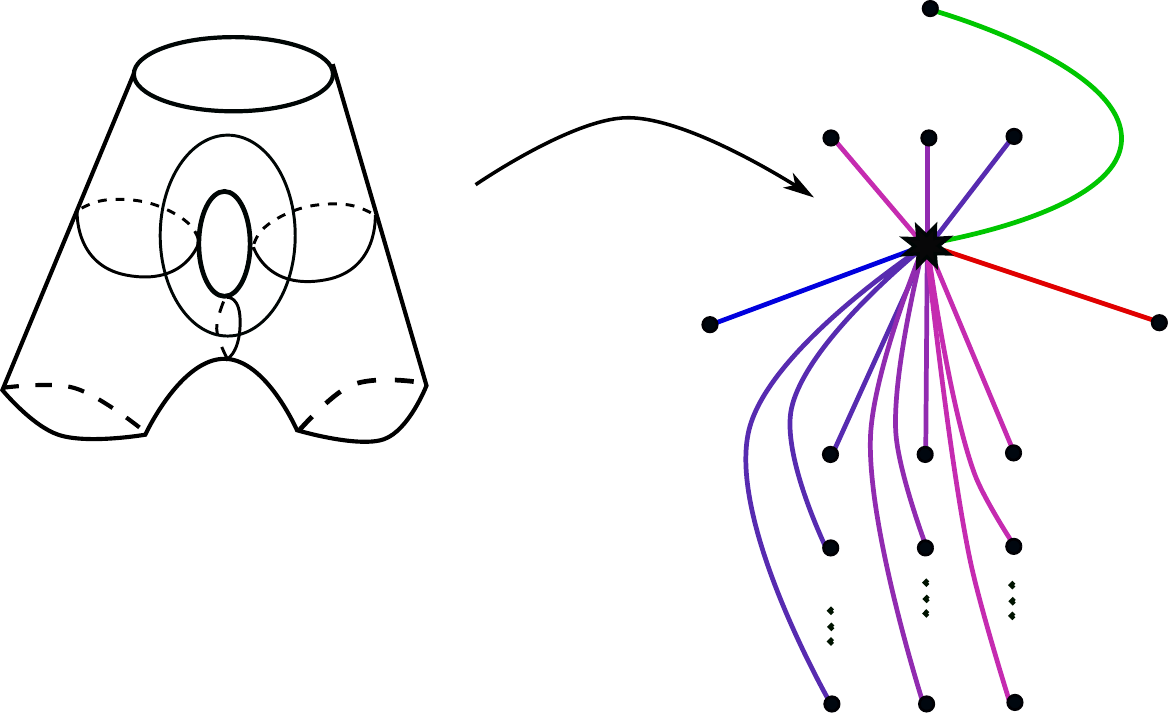}};
\node at (4,0.5) {$V_{-2}$};
\node at (3.6,-0.5) {$V_{p+q+r}$};
\node at (-2.9,1.45) {$T$};
\node at (-3.8,0.8) {$P$};
\node at (-2.5,-0.5) {$Q$};
\node at (-1,0.8) {$R$};
 \end{tikzpicture}\caption{A Lefschetz fibration of $T_{p,q,r}$ whose vanishing cycles are Dehn twists of the four curves $T,P,Q,R$ on the regular fiber.}\label{fig:lefschetzTpqr}
\end{figure}
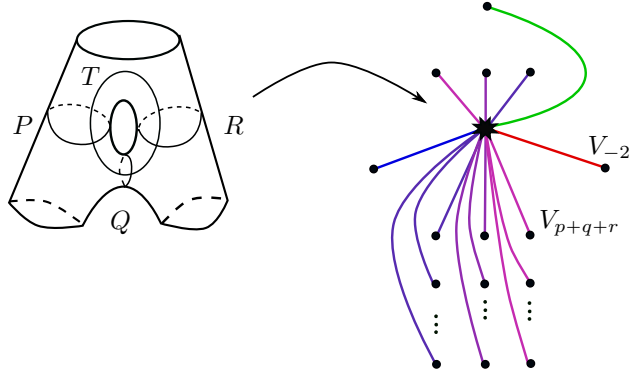

\newtheorem*{prop:Tpqr}{Proposition~\ref{prop:Tpqr}}
\begin{prop:Tpqr} The Weinstein domains 
$$T_{p,q,r}=\{(x,y,z)\in \C^3~|~x^p+y^q+z^r+xyz=1\}$$
have the Weinstein handlebody diagram shown in Figure~\ref{fig:Tpqrhandle} for $p,q,r\geq0$. 
\end{prop:Tpqr}

\begin{proof}
We start with the Lefschetz fibration $\Xi:T_{p,q,r}\rightarrow \C$ from~\cite{keating2} which has vanishing cycles that can be written as Dehn twists of the four curves four curves $T,P,Q,R$ on the regular fiber $\Xi^{-1}(0)$ shown in Figure~\ref{fig:lefschetzTpqr}. These vanishing cycles are:
\begin{align*}V_{-2}&=\tau_P\tau_Q\tau_R T, ~V_{-1}=T,~ V_{0}=P,~V_1=Q,~V_2=R,~
V_3=T,\\V_4&=\ldots =V_{p+3}=P,~ V_{p+4}=\ldots=V_{p+q+3}=Q,~V_{p+q+4}= \ldots=V_{p+q+r}=R.\end{align*}

We apply three Hurwitz moves and obtain the following collection of vanishing cycles:
\begin{align*}V_{-2}&=\tau_P\tau_Q\tau_R T, ~V_{-1}=T,~ V_{0}=\tau^{-1}_P\tau^{-1}_Q\tau^{-1}_R T,\\V_1&=
\ldots= V_p=P, ~V_{p+1}=\ldots V_{p+q}=Q,~V_{p+q+1}=\ldots= V_{p+q+r}=R.\end{align*}

We now apply the affine dictionary and obtain the Weinstein handlebody diagram shown in Figure~\ref{fig:Tpqrhandle}. First we use the Lefschetz fibration on the regular fiber $\Xi^{-1}(0)$ where the Lagrangian skeleton of consists of Lagrangian spheres $T,P,Q$, and $R$. We already know how to write the vanishing cycles of $\Xi$ as Dehn twists of these four curves $T,P,Q,R$, so we are in Step $5$ of the affine dictionary and can use Proposition $2.23$ of~\cite{Casals_Murphy} to draw the front projection of their Legendrian lifts. Note that the order of the vanishing cycles determines the Reeb height of the Legendrian lifts.
\end{proof}

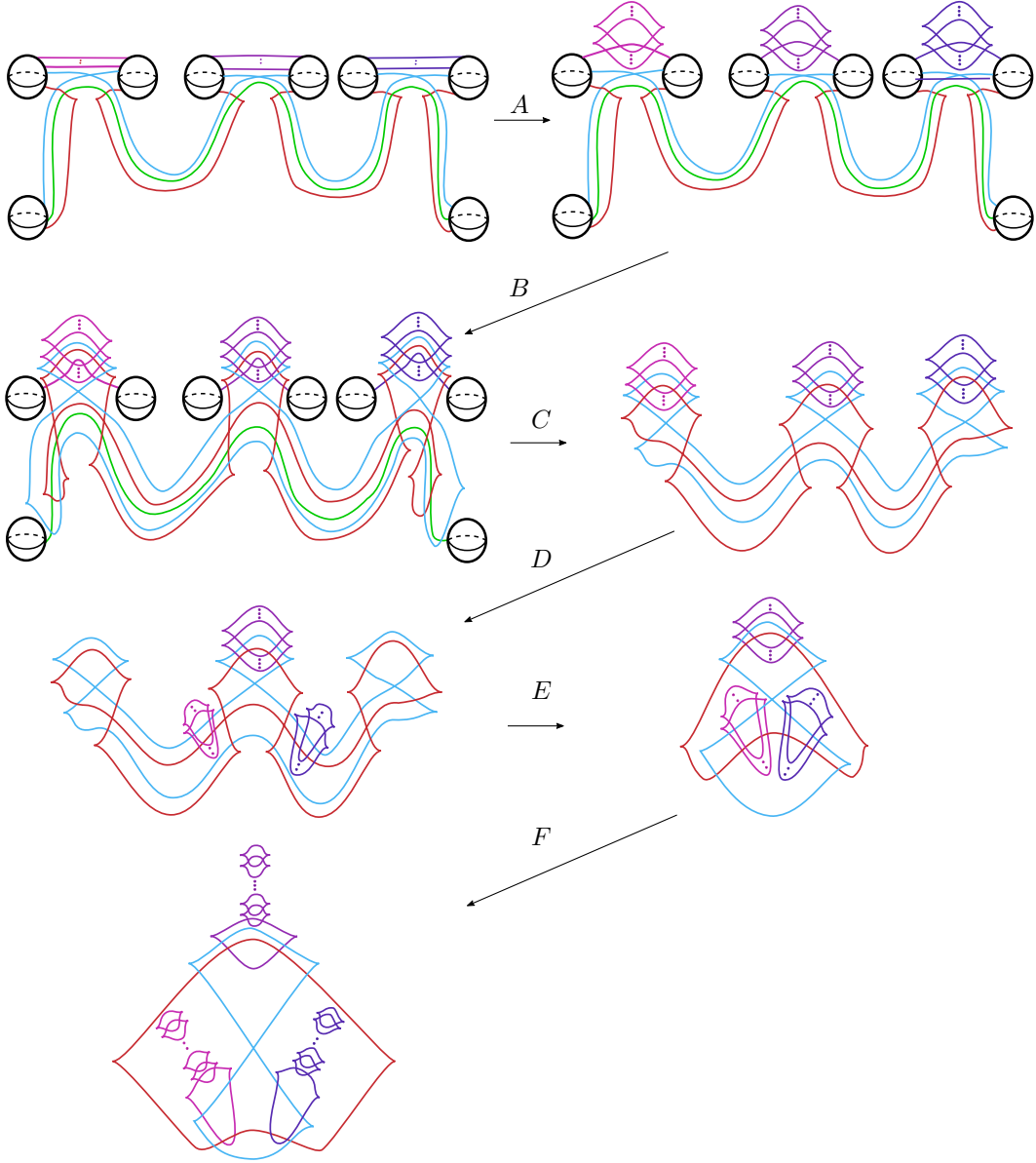
\begin{figure}[h!]
    \centering
\begin{tikzpicture}
\node[inner sep=0] at (0,0) {\includegraphics[width=0.9\linewidth]{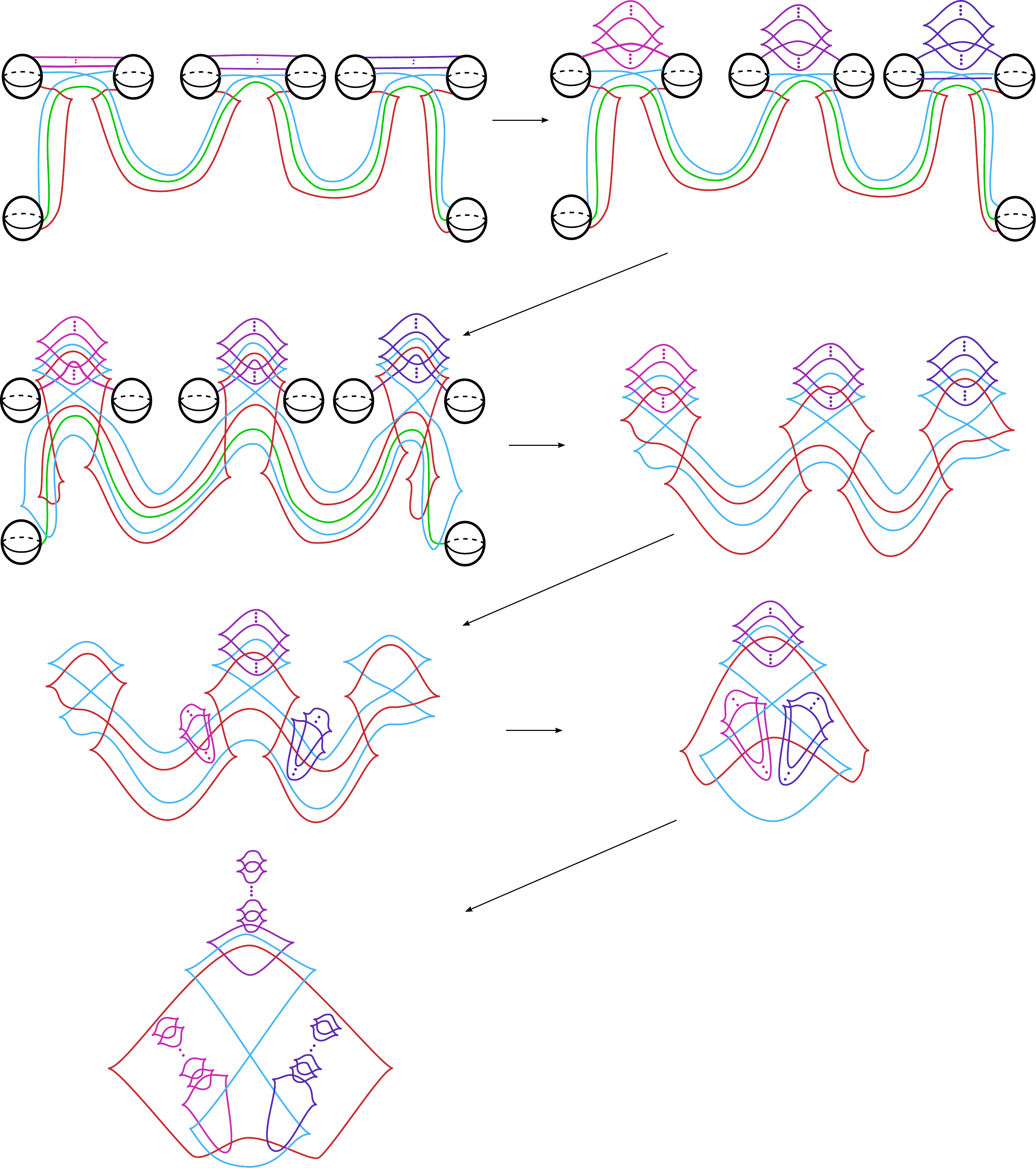}};
\node at (0,6.5) {$A$};
\node at (0, 4) {$B$};
\node at (0.3,2.2) {$C$};
\node at (0.3,.3) {$D$};
\node at (0.3,-1.5) {$E$};
\node at (0.3,-3.5) {$F$};
\end{tikzpicture}
\caption{We simplify the Weinstein handlebody diagram of $T_{p,q,r}$ for $p,q,r\geq1$ as follows: $(A)-(B)$ perform various handle-slides; $(C)$ cancel out the $1$-handles; $(D)-(E)$ perform Legendrian Reidemeister moves; $(F)$ use the fact that one can handle-slide an $N$-copy of the unknot to a chain of $N$ unknots.}\label{fig:Tpqr1}
\end{figure}

\begin{figure}
\centering
\begin{tikzpicture}
\node[inner sep=0] at (0,0) {\includegraphics[width=0.85\linewidth]{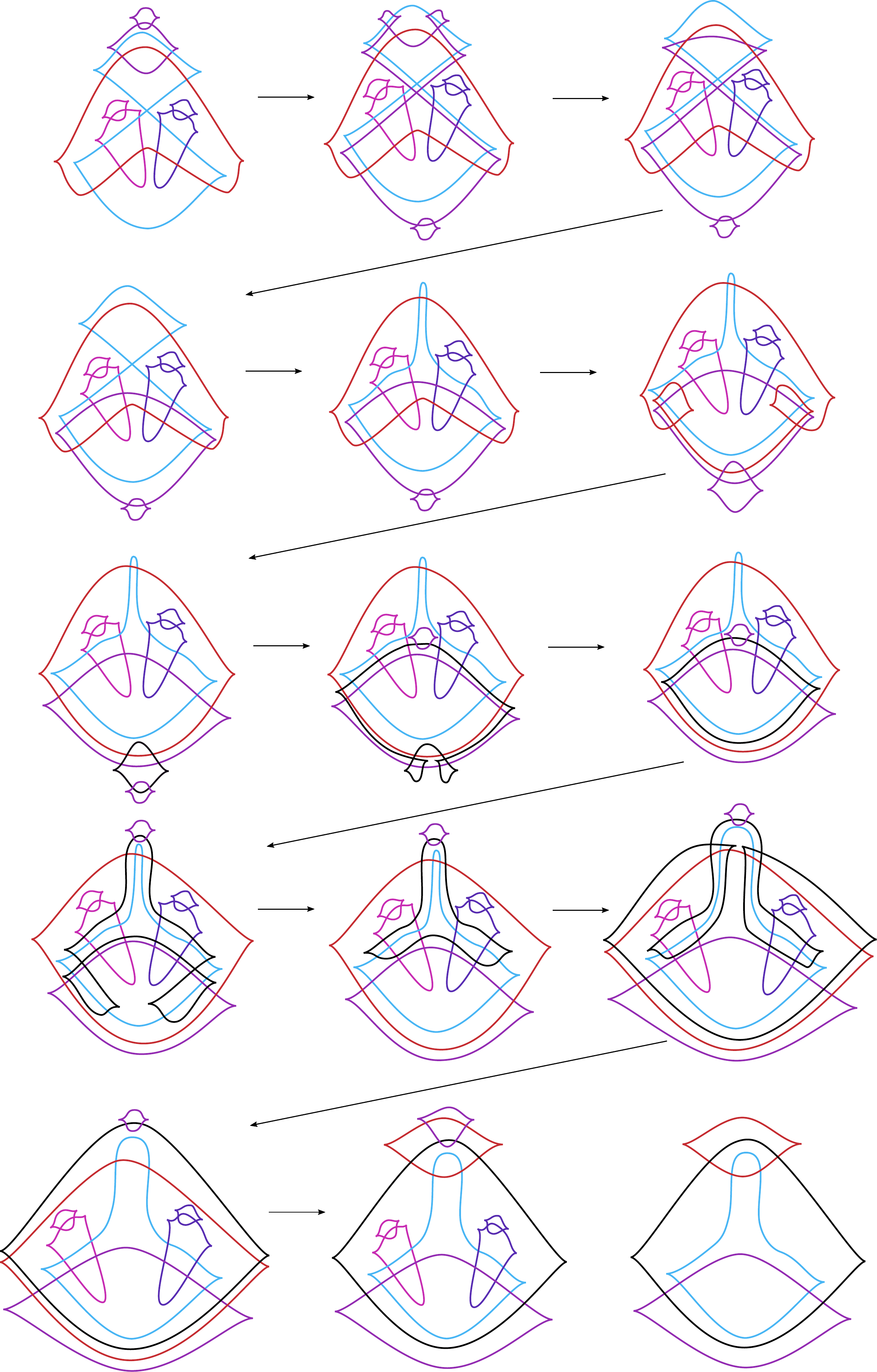}};
\node at (-5.8,8.6) {$V_{-1}$};
\node at (-5.3,9.7) {$V_{p+1}$};
\node at (-5.3,10.2) {$V_{p+2}$};
\node at (-3.3, 6.8) {$V_0$};
\node at (-2.4,9.2) {$A$};
\node at (2, 9.2) {$B$};
\node at (-1.6,6.4) {$C$};
\node at (-2.4,5) {$D$};
\node at (2, 5) {$E$};
\node at (-1.6,2.5) {$F$};
\node at (-2.4,0.9) {$G$};
\node at (2,0.9) {$H$};
\node at (-1.6,-1.8) {$I$};
\node at (-2.4,-3) {$J$};
\node at (2,-3) {$K$};
\node at (-1.6,-5.9) {$L$};
\node at (-2.1,-7.5) {$M$};
\node at (2.4,-8) {$\supseteq$};
\end{tikzpicture}
\caption{We find a Weinstein handlebody diagram of $T_{p,q,r}$ that contains as a sublink $\Lambda(\beta_{22})$ when $p,r\geq1$ and $q\geq 3$. The steps $(A), (E), (F), (I), (K), (M)$ are handle slides, and the steps $(B)-(D), (F), (H), (J), (L)$ are Legendrian Reidemeister moves.}\label{fig:Tpqr_link2}
\end{figure}
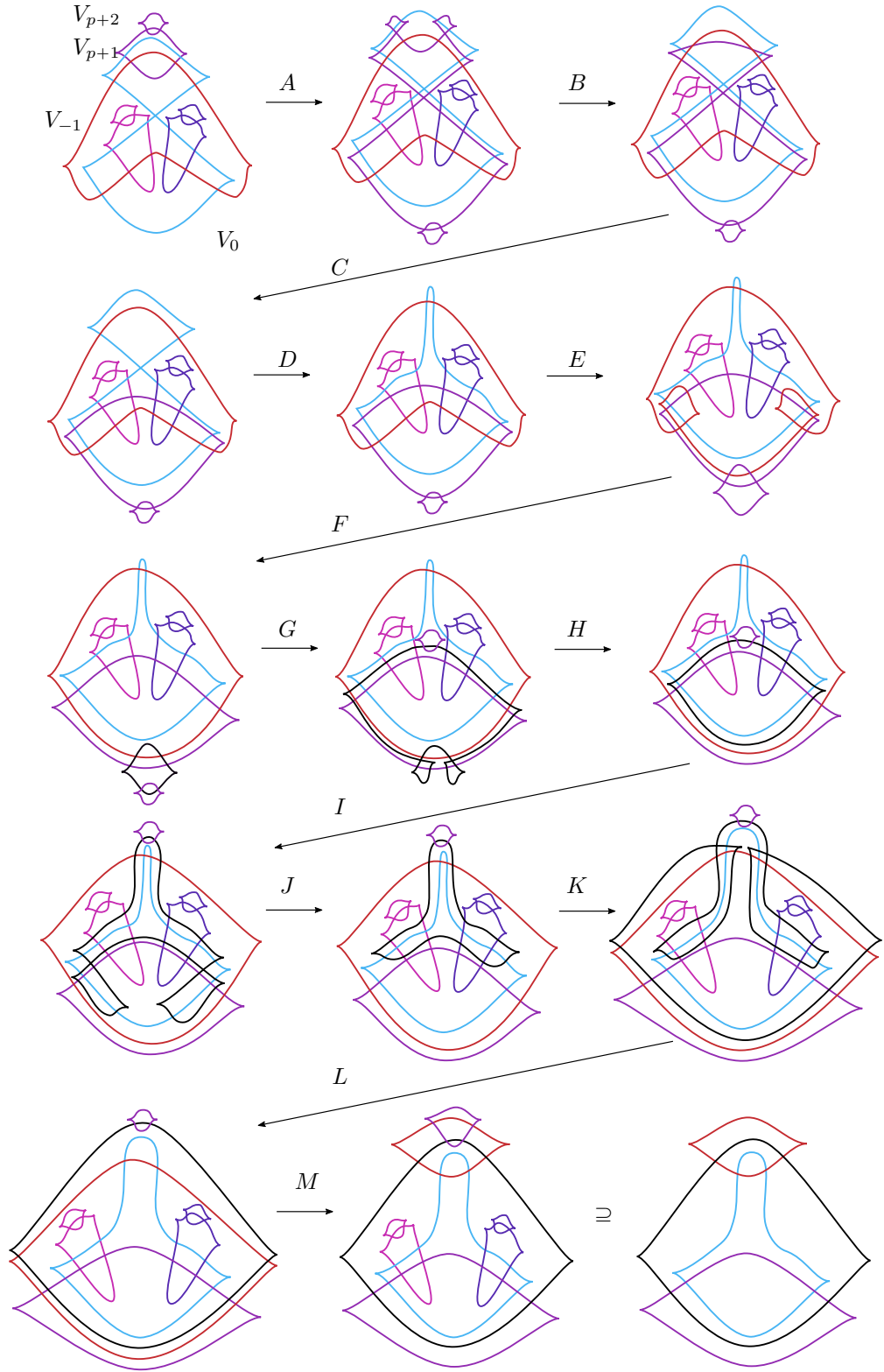

\begin{figure}
    \centering
\begin{tikzpicture}
	\node[inner sep=0] at (0,0) {\includegraphics[width=0.9\linewidth]{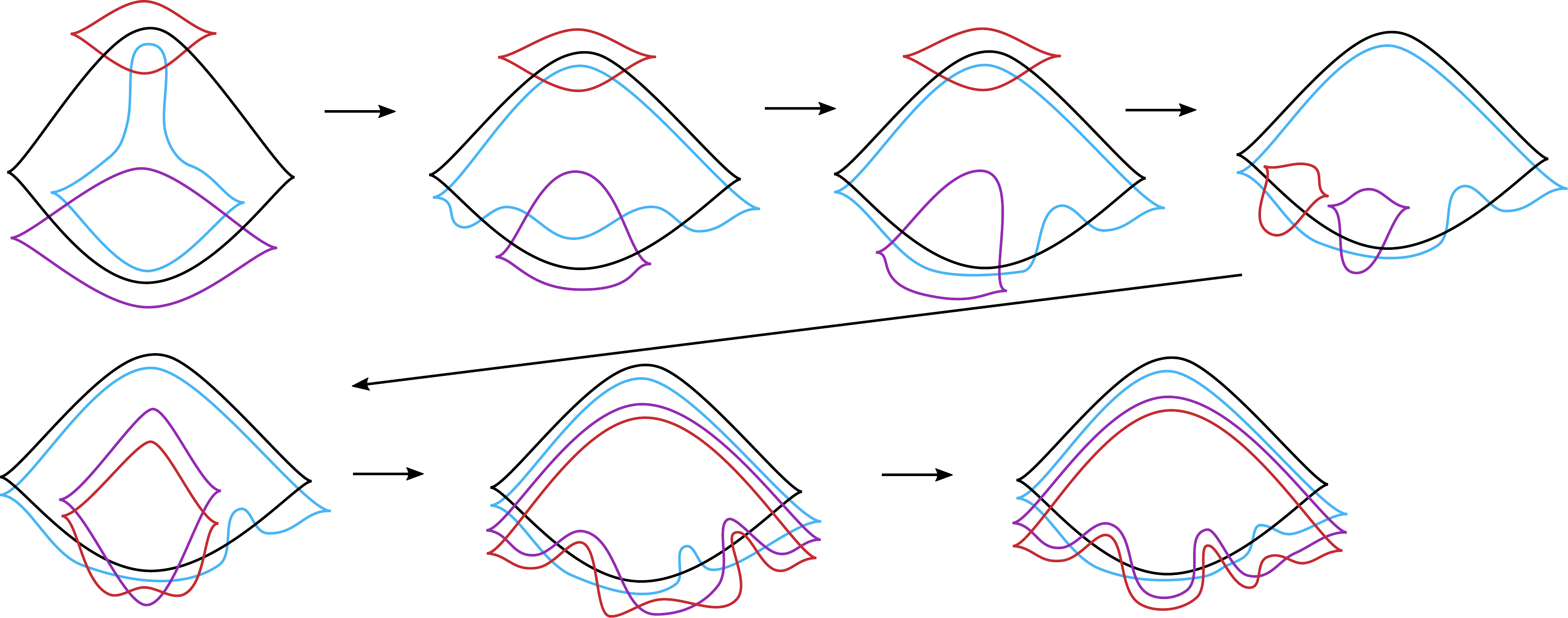}};
\node at (-3.9,2.1) {$R2$};
\node at (0.1, 2.1) {$R3$};
\node at (3.5, 2.1) {$R1-R3$};
\node at (0.1, 0.2) {$R2-R3$};
\node at (-3.3, -1.2) {$R2-R3$};
\node at (1.1, -1.2) {$R3$};
\end{tikzpicture}
\caption{A sequence of Legendrian Reidemeister moves on the Legendrian link $\Lambda(\beta_{22})$, the $(-1)$ closure of $(\sigma_2\sigma_3\sigma_1\sigma_2)^4\sigma_3^2\sigma_1^2 \in Br^+_{4}$. We indicate which types of Legendrian Reidemeister moves are used at each intermediate step.}\label{fig:sublink2r}
\end{figure}

\newtheorem*{prop:sublink}{Theorem~\ref{prop:sublink}}
\begin{prop:sublink}
For any $p,r\geq1$, and $q\geq 3$, the Weinstein $4$-manifold $T_{p,q,r}$ contains infinitely many Hamiltonian non-isotopic exact Maslov-$0$ Lagrangian tori that are all smoothly isotopic and primitive in homology.
\end{prop:sublink}

\begin{proof}
We start with the Weinstein handlebody diagram of $T_{p,q,r}$ obtained in Proposition~\ref{prop:Tpqr} and shown in Figure~\ref{fig:Tpqr1}. After performing Legendrian isotopy, Legendrian handle slides and handle cancellations, as described in Figures~\ref{fig:Tpqr1},~\ref{fig:Tpqr_link2},~and~\ref{fig:sublink2r} we find a Weinstein handlebody diagram of $T_{p,q,r}$ that contains a sublink that is Legendrian isotopic to $\Lambda(\beta_{22})$ if $p,r\geq1$, and $q\geq 3$. It follows from Corollary~\ref{thm:weinsteinsublink}, that if $p,r\geq1$, and $q\geq 3$, then $T_{p,q,r}$ contains infinitely many Hamiltonian non-isotopic Lagrangian tori that are smoothly isotopic and primitive in homology. 

We will now give a brief summary of the Legendrian isotopies, handle-slides and handle cancellations employed in Figures~\ref{fig:Tpqr1},~\ref{fig:Tpqr_link2} and~\ref{fig:sublink2r}. For simplicity of notation we do not change the notation for the attaching sphere of a $2$-handle after handle slides, although they are distinct as Legendrians. 
We first simplify the Weinstein handlebody diagram of $T_{p,q,r}$ as shown in Figure~\ref{fig:Tpqr1}: $(A)-(B)$ perform various handle-slides; $(C)$ cancel out the $1$-handles; $(D)-(E)$ perform Legendrian Reidemeister moves; $(F)$ use the fact that one can handle-slide an $N$-copy of the unknot to a chain of $N$ unknots. As soon as one cancel the $1$-handles one can check that $T_{p,q,r}$ has the correct intersection form for $p,q,r\geq 1$. Furthermore, one can also check that $T_{0,0,0}$ is Weinstein homotopy equivalent to $D^*T^2$, that is that there exists a sequence of handle-slides and Legendrian isotopies that takes the Weinstein handlebody diagram of $T_{0,0,0}$ to the standard Weinstein handlebody diagram of $D^*T^2$. We are now ready to use the Weinstein handlebody diagram of $T_{p,q,r}$ to find closed exact Lagrangian surfaces in $T_{p,q,r}$. 

From here on out we do no include the vanishing cycles that are chains of unknots as they do not play an active role in the computation. Next we simplify the Weinstein handlebody diagram of $T_{p,q,r}$ as shown in Figure~\ref{fig:Tpqr_link2}:
$(A)$ Handle slide $V_{p+1}$ under $V_0$; $(B)-(D)$ perform Legendrian Reidemeister moves; $(E)$ handle slide $V_{-1}$ over $V_{p+1}$; $(F)$ perform Legendrian Reidemeister moves; $(G)$ handle slide $V_{p+2}$ under $V_{p+1}$; $(H)$ perform Legendrian Reidemeister moves; $(I)$ handle slide $V_{p+2}$ over $V_0$; $(J)$ perform Legendrian Reidemeister moves; $(K)$ handle slide $V_{p+2}$ over $V_{-1}$; $(L)$ perform Legendrian Reidemeister moves; $(M)$ handle slide $V_{-1}$ over $V_{p+2}$. Finally, as described by Figure~\ref{fig:sublink2r} we perform a series of Legendrian Reidemeister moves on a sublink from the last diagram in Figure~\ref{fig:Tpqr_link2} to prove that it is Legendrian isotopic to $\Lambda(\beta_{22})$, the $(-1)$ closure of the braid $(\sigma_2\sigma_1\sigma_3\sigma_2)^4\sigma_1^2\sigma_3^2\in Br_4^+$.  
\end{proof}

There are also important relations between singularities that one can consider, one of which is whether a singularity is adjacent to another. A singularity $[f]$ is \textbf{adjacent} to $[g]$ if there exists an arbitrarily small perturbation $p$ such that $[f+p]=[g]$.

\begin{lemma}[Lemma $9.9$ in~\cite{keating2}]\label{lem:adjacent_embedding}
Suppose $[f]$ and $[g]$ are singularities, such that $[f]$ is adjacent to $[g]$. Then there exists an exact symplectic embedding from a non-completed Milnor fiber of $g$ into a completed Milnor fiber of $f$.
\end{lemma}

Each of Arnold's exceptional singularities is adjacent to one of the three parabolic singularities. For $p'\geq p, q'\geq q,$ and $r'\geq r$, $T_{p',q',r'}$ is adjacent to $T_{p,q,r}$. More generally, Durfee showed that:

\begin{theorem}[Durfee~\cite{durfee}]\label{thm:durfee} Any positive modality hypersurface singularity is adjacent to a modality one hypersurface singularity.
\end{theorem}

\newtheorem*{cor:inftori}{Corollary~\ref{cor:inftori}}
\begin{cor:inftori}
Suppose that $M_f$ is the Milnor fiber of a positive modality isolated hypersurface singularity $f$, then $M_f$ contains infinitely many Hamiltonian non-isotopic exact Maslov-$0$ Lagrangian tori that are all smoothly isotopic.
\end{cor:inftori}
\begin{proof}
 By Theorem~\ref{prop:sublink}, the Milnor fiber of any of the three parabolic singularities, that is the $T_{p,q,r}$ singularity for $(p,q,r)\in\{(3,3,3), (2,4,4), (2,3,6)\}$, has infinitely many Hamiltonian non-isotopic exact Maslov-$0$ Lagrangian tori that are smoothly isotopic. All positive modality singularities $f$ are adjacent to a unimodular singularity by Theorem~\ref{thm:durfee}, and all unimodular singularities are adjacent to at least one of the three parabolic singularities. By Lemma~\ref{lem:adjacent_embedding}, there then exists an exact symplectic embedding of the non-completed Milnor fiber of such $T_{p,q,r}$ singularities into $M_f$. %Under an embedding such as the one given in Lemma~\ref{lem:adjacent_embedding}, vanishing cycles of $g$ get mapped to vanishing cycles for $f$ up to compactly supported Hamiltonian isotopy. Thus, the tori are primitive in homology for all 
\end{proof}

\section{Criterion for the nonvanishing of the Symplectic Homology of a Weinstein $4$-manifold.}\label{sec:criterion}

Bourgeois, Ekholm and Eliashberg in~\cite{BEE, Ekholm} gave a relation between Legendrian invariants of the critical attaching Legendrian spheres of $2$-handles and symplectic invariants of the resulting Weinstein manifold. In particular, they showed that
$$S \mathbf{H}(X_{\Lambda})=L\mathbf{H}^{H_0}(\Lambda)$$
where $L\mathbf{H}^{H_0}(\Lambda)$ is the homology of Hochschild like complex associated to the Legendrian contact homology differential graded algebra of $\Lambda$ over $\Q$. Leverson in~\cite[Corollary 1.5]{Leverson} used this result to show that if $\Lambda$ has a graded augmentation, then the symplectic homology $S\mathbf{H}(X_{\Lambda})$ does not vanish. We provide a generalization of this non-vanishing criterion, where now for Legendrian links $\Lambda$ that have at least one sublink with a graded augmentation or a graded representation, $S\mathbf{H}(X_{\Lambda})\neq 0$. See Figure~\ref{fig:examples_sh} for three such Legendrians. The first two Legendrian links shown on the left of Figure~\ref{fig:examples_sh} do not have a graded augmentation but have a sublink (the max-tb unknot link component) that does. The Legendrian knot $m(10_{132})$ whose front is shown on the right of Figure~\ref{fig:examples_sh} has a graded representation but no graded augmentations or finite graded representations~\cite[Section 2.2]{Sivek} (the proof is written over $\Z_2$ but also holds with minor modifications over $\Z$). 

\newtheorem*{thm:criterion}{Theorem~\ref{thm:criterion}}
\begin{thm:criterion}
Let $X_{\Lambda}$ be the Weinstein $4$-manifold resulting from attaching $2$-handles along a Legendrian link $\Lambda \subset (\#^m(S^1\times S^2), \xi_{std})$ with $n$ link components. If there is any sublink $\Lambda'$ with $l$ link components for $l\leq n$, such that its differential graded algebra has a graded representation 
$$\rho: (\mathcal{A}(\Lambda';\Z[t_1^{\pm 1} \ldots, t_l^{\pm 1}]), \partial) \rightarrow End(V)$$
where $V$ is a vector space over $\Q$ and $\rho(t_k)=-Id$ for $k=1,\ldots, l$, then $S\mathbf{H}(X_{\Lambda})\neq0$.
\end{thm:criterion}

\begin{figure}[t!]
    \centering
\begin{tikzpicture}
	\node[inner sep=0] at (0,0) {\includegraphics[width=0.5\linewidth]{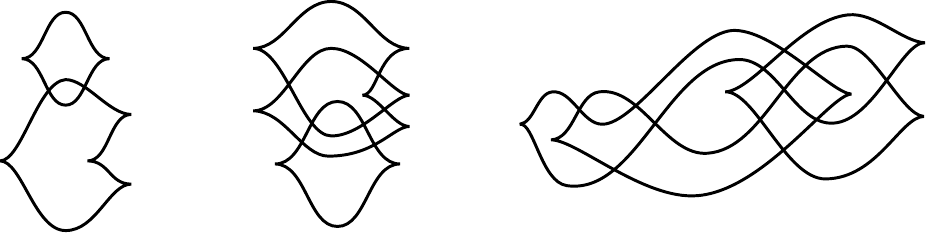}};
\end{tikzpicture}
\caption{Weinstein handlebody diagrams of Weinstein $4$-manifolds with non-vanishing symplectic cohomology. The two Legendrian links on the left have one link component that does not have a graded augmentation and one link component (the max-tb unknot) that does have a graded augmentation. The $m(10_{132})$ Legendrian knot on the right has a graded representation but no graded augmentations.}\label{fig:examples_sh}
\end{figure}

There is a fair amount of work that goes into defining a chain complex and differential of $L\mathbf{H}^{H_0}(\Lambda)$ in~\cite{BEE}. For the readers convenience we recall the relevant definitions and notation, with the slight modification that chain complexes are denoted with a $C$ instead of $CH$.

\textbf{The generators of $LC^{H_0}(\Lambda)$:} Consider the Legendrian contact homology differential graded algebra $(\mathcal{A}(\Lambda; \Z), \partial)$, with the null cobordant spin structure on each $\Lambda^{(i)}$. For the purposes of this section, we assume the dg-algebra also had generators $e_i$, where $e_i$ is an idempotent element associated to the link component $\Lambda^{(i)}$ and which we refer to as an empty Reeb chord. Moreover, we also require the empty Reeb chords to satisfy the following relations:
$$ e_1+\cdots +e_n=1,~~ e_i\cdot e_j =\delta_{ij}$$
and
$$
e_i\cdot c =\begin{cases}
c & \text{If}~c\in \mathcal{R}_{ij}(\Lambda)~\text{for some}~j\\
0 & \text{otherwise}
\end{cases}~~c\cdot e_j=\begin{cases}
c & \text{If}~c\in \mathcal{R}_{ij}(\Lambda)~\text{for some}~i\\
0 & \text{otherwise}
\end{cases}
$$
The subalgebra $LCO(\Lambda)\subset \mathcal{A}(\Lambda)$ is generated by cyclically composable monomials of Reeb chords, while $LCO^{+}(\Lambda):=LCO(\Lambda)/ \langle e_1,\ldots,e_n\rangle$ is the subalgebra generated by non-trivial cyclically composable monomials of non-empty Reeb chords. Let $\widecheck{LCO}^{+}(\Lambda):=LCO^{+}(\Lambda)$ and $\widehat{LCO}^{+}(\Lambda):=LCO^{+}(\Lambda)[1]$. For any element $w=c_1\cdots c_j\in LHO^{+}(\Lambda)$, write the corresponding elements $\check{w} \in \widecheck{LCO}^{+}(\Lambda)$ and $\hat{w}\in \widehat{LCO}^{+}(\Lambda)$ as follows $\check{w}=\check{c_1}\cdots c_n$ and $\hat{w}=\hat{c_1}\cdots c_n$. 
The chain complex is then
$$LC^{H_0}(\Lambda)=\widecheck{LCO}^{+}(\Lambda)\oplus \widehat{LCO}^{+}(\Lambda)\oplus\Q\langle \tau_1, \dots, \tau_n\rangle,$$
 where each $\tau_i$ has grading $0$. Each $\tau_i$ is in bijective correspondence with each $e_i$ and thus with each Legendrian link component $\Lambda^{(i)}$.

\textbf{The differential $d_{H_0}$:} Given an element $(\check{w},\hat{v}, \sum_{i=1}^n m_i \tau_i)$ in $LC^{H_0}(\Lambda)$, where $\hat{v}=\hat{c_1}\cdots c_k$ and $m_i\in \Q$, the differential is given by
$$
d_{H_0}(\check{w},\hat{v},\sum_{i=1}^n m_i \tau_i):=(\check{d}_{LCO^{+}}(\check{w})+\check{c_1}\cdots c_k -c_1 \cdots \check{c_k}, \hat{d}_{LCO^{+}}(\hat{v}), \delta_{H_0}(\check{w}+\hat{v})).
$$
The differentials $\check{d}_{LCO^{+}}$ and $\hat{d}_{LCO^{+}}$ are the differential of $\mathcal{A}(\Lambda)$ restricted to $\widecheck{LCO}^{+}(\Lambda)$ and $\widehat{LCO}^{+}(\Lambda)$ respectively. The differential $\delta_{H_0}$ is defined as follows:
\begin{itemize}\item if $c$ is a non-empty pure Reeb chord beginning and ending on $\Lambda^{(i)}$, then
$\delta_{H_0}(\check{c})=n_{c_i}\tau_i$, where $n_{c_i}$ is the count of the zero dimensional moduli space of holomorphic disks asymptotic at $+\infty$ to $c$; 
\item if $w$ is a nonlinear monomial, then $\delta_{H_0}(\check{w})=0$;
\item for any element $\hat{v}$, $\delta_{H_0}(\hat{v}):=0$.
\end{itemize}

\begin{proof}

Our goal is to construct a representation $\tilde{\rho}: (LC^{H_0}(\Lambda), d_{H_0}) \rightarrow (End(V^{\oplus n}),0)$ such that $\tilde{\rho} \circ d_{H_0}=0$. We use such a representation $\tilde{\rho}$ to show that $d_{H_0}(\tau_k)=0$ and $\tau_k \notin Im(d_{H_0})$ for $k=1, \ldots, n$. We can then conclude that $LC\mathbf{H}^{H_0}(\Lambda)=S\mathbf{H}(X_{\Lambda})$ is nonzero. We will label the link components in $\Lambda=\Lambda^{(1)}\cup \cdots \cup \Lambda^{(n)}$ so that $\Lambda'$ is given by the first $l$ link components, $\Lambda^{(1)} \cup \cdots \cup \Lambda^{(l)}$.

First define a graded map $\rho': \mathcal{A}(\Lambda)\oplus\Q\langle \tau_1, \ldots, \tau_n\rangle\rightarrow End(V^{\oplus n})$ as follows. 
For $1\leq k \leq l$
$$\rho'(t_k)_{ij}=\begin{cases}
Id & i=j=k\\
0& \text{otherwise}.\\
\end{cases}$$ 
For $k>l$, $$\rho'(\tau_k)_{ij}=0.$$\\
For $1\leq k \leq l$
$$\rho'(\tau_k)_{ij}=\begin{cases}
Id & i=j=k\\
0& \text{otherwise}.
\end{cases}$$ 
If $c\in \mathcal{R}_{kr}(\Lambda)$ is a Reeb chord beginning on $\Lambda_k$ and ending on $\Lambda_r$ and $1\leq k,r\leq l$, then
$$\rho'(c)_{ij}=\begin{cases} 
     \rho(c) & i=k,~j=r\\
     0 & \text{otherwise}.
   \end{cases}$$
If $c\in \mathcal{R}_{kr}(\Lambda)$ but $l<k$ or $l<r$, then $$\rho'(c)_{ij}=0.$$
Thus, for any element $(\check{w},\hat{v}, \sum_{i=1}^n m_i \tau_i)\in LC^{H_0}(\Lambda)$, where $\hat{v}=\hat{c_1}\cdots c_k$,and $m_i\in \Q$, define
$$\tilde{\rho}(\check{w},\hat{v}, \sum_{i=1}^n m_i \tau_i)=\rho'(w)+\rho'(m_i \tau_i).$$
We need to check that $\tilde{\rho} \circ d_{H_0}=0$:
\begin{align*}
\tilde{\rho} (d_{H_0}(\check{w},\hat{v}, \sum_{i=1}^n m_i\tau_i))&=\tilde{\rho}(dw+\check{c_1}\cdots c_j -c_1 \cdots \check{c_j}, \hat{d}_{LCO^{+}}(\hat{v}), \delta_{H_0}(\check{w}+\hat{v}))\\
&=\rho'(\check{d}_{LCO^{+}}(\check{w})+c_1\cdots c_k -c_1 \cdots c_k)+\rho'(\delta_{H_0}(\check{w}+\hat{v}))\\
&=\rho'(\check{d}_{LCO^{+}}(\check{w})+\delta_{H_0}(\check{w})).
\end{align*}
By definition of $\delta_{H_0}$ there are two cases to consider. If $w$ is a nonlinear monomial that begins and ends on the link component $\Lambda_k$ for $1\leq k\leq n$, then 
$$\tilde{\rho} (d_{H_0}(\check{w},\hat{v}, \sum_{i=1}^n m_i \tau_i))=\rho'(\check{d}_{LCO^{+}}(\check{w})+0),$$
and by definition of $\rho'$
$$\rho'(\check{d}_{LCO^{+}}(\check{w}))_{ij}=\begin{cases}
\rho(dw)& i=j=k\\
0&\text{otherwise}
\end{cases}$$
since $\check{d}_{LCO^{+}}=d|_{\widecheck{LCO}^{+}}$. If $w=c$ is a Reeb chord beginning and ending on the link component $\Lambda_k$ for $1\leq k\leq n$, and because we can identify $\tau_i$ with the empty Reeb word $e_i\in LCO(\Lambda)$, we can conclude that
$$\tilde{\rho} (d_{H_0}(\check{w},\hat{v}, \sum_{i=1}^n m_i \tau_i))=\rho'(\check{d}_{LCO^{+}}(\check{c})+n_{c_k}\tau_k)=\rho'(d_{LCO}(c)).$$
Then,
$$\rho'(d_{LCO}(c))_{ij}=\begin{cases}
\rho(d(c))& i=j=k\\
0&\text{otherwise}.
\end{cases}$$
Since we started with a representation $\rho$ such that $\rho \circ \partial =0$, we can conclude that $\tilde{\rho}\circ d_{H_0}=0.$ Finally, we observe that $\tau_k$ for $k=1, \ldots, n$ is a cycle because $d_{H_0}(0,0,\tau_k)=(0,0,0)$. Moreover, $\tilde{\rho}(\tau_k)=\rho(t_k)=-Id\neq 0$, and $\tilde{\rho} \circ d_{H_0}=0$, so we can conclude that $\tau_k\notin Im(d_{H_0})$ for $k=1,\ldots, n$.
\end{proof}

%Flexible Weinstein manifolds abide by an h-principle~\cite{Cieliebak_eliashberg} and it is often difficult to determine whether a Weinstein manifold is flexible or not. A Weinstein manifold is said to be flexible if the critical handles areattached along loose Legendrian submanifolds. Since flexible Weinstein manifolds have vanishing symplectic homology, we can then immediately conclude the following.

\newcommand{\noopsort}[1]{} \newcommand{\printfirst}[2]{#1}
  \newcommand{\singleletter}[1]{#1} \newcommand{\switchargs}[2]{#2#1}

\end{document}